\documentclass[priprent]{elsarticle}
\usepackage{graphicx, epstopdf}
\usepackage{amssymb,amsmath}
\usepackage[toc,page,title,titletoc,header]{appendix}
\usepackage{amsbsy}
\usepackage{amsthm}
\usepackage{multirow}
\usepackage{indentfirst}
\usepackage{natbib}
\usepackage{amscd}
\numberwithin{equation}{section}

\newtheorem{Theorem}{Theorem}[section]

\newtheorem{Lemma}{Lemma}[section]

\newtheorem{Assumption-Notation}[Theorem]{Assumption-Notation}

\newtheorem{Remark}{Remark}[section]
\newtheorem{Corollary}{Corollary}[section]

\newtheorem{lem}{Lemma}

\newtheorem{rem}{Remark}

\makeatletter
\renewcommand\@biblabel[1]{}
\renewenvironment{thebibliography}[1]
{\section*{\refname}%
\@mkboth{\MakeUppercase\refname}{\MakeUppercase\refname}%
\list{\@biblabel{\@arabic\c@enumiv}}%
{\settowidth\labelwidth{\@biblabel{#1}}%
\leftmargin\labelwidth
\advance\leftmargin\labelsep
\advance\leftmargin by 2em%
\itemindent -2em%
\@openbib@code
\usecounter{enumiv}%
\let\p@enumiv\@empty
\renewcommand\theenumiv{\@arabic\c@enumiv}}%
\sloppy
\clubpenalty4000
\@clubpenalty \clubpenalty
\widowpenalty4000%
\sfcode`\.\@m}
{\def\@noitemerr
{\@latex@warning{Empty `thebibliography' environment}}%
\endlist}
\makeatother

\long\def\symbolfootnote[#1]#2{\begingroup
\def\thefootnote{\fnsymbol{footnote}}\footnote[#1]{#2}\endgroup}

\journal{Insurance: Mathematics and Economics}
\begin{document}
\begin{frontmatter}
\title{Pricing variable annuities with multi-layer expense strategy}

 \author{Jiang Zhou }
 \ead{1101110056@pku.edu.cn}
 \author{Lan Wu}
\address{School of Mathematical Sciences, Peking University, Beijing 100871, P.R.China}

\begin{abstract}{We study the problem of pricing variable annuities with a multi-layer expense strategy, under which the insurer charges fees from the policyholder's account only when the account value lies in some pre-specified disjoint intervals, where on each pre-specified interval, the fee rate is fixed and can be different from that on other interval. We model the asset that is the underlying fund of the variable annuity by a hyper-exponential jump diffusion process. Theoretically, for a jump diffusion process with hyper-exponential jumps and three-valued drift, we obtain expressions for the Laplace transforms of its distribution and its occupation times, i.e., the time that it spends below or above a pre-specified level. With these results, we derive closed-form formulas to determine the fair fee rate. Moreover, the total fees that will be collected by the insurer and the total time of deducting fees are also computed. In addition, some numerical examples are presented to illustrate our results.}
\end{abstract}

\end{frontmatter}
\section{Introduction}
A variable annuity (VA) is an equity-linked life insurance product with minimum guarantee on the death or maturity benefits. Generally, the insured chooses a mutual fund according to his/her risk preference, and contributes an initial premium to invest in the chosen fund. A fascinating feature of VAs, which makes them different from mutual funds, is that they usually provide some minimum guaranteed payoffs.  There are many kinds of guarantees embedded in VAs (see, e.g., Bauer et al. (2008)) and many papers investigating the problem of pricing and hedging VAs, see for example, Gerber and Shiu (2003) and Bacinello et al. (2011). In this paper, we focus on the pricing of a variable annuity with level Guaranteed Minimum Maturity Benefits (GMMBs). But we remark that our results can be extended to price a variable annuity with level Guaranteed Minimum Death Benefits (GMDBs).

In general, for a VA, fees for the provided guarantees  are deducted continuously from the policyholder's account during its lifetime by a fixed rate. If the VA expires earlier than expected (one main cause is that the policyholder lapses the policy), it is possible that fees collected by the insurer may be not enough to cover the guarantees. Therefore, how to reduce the surrender rate is an important question to the insurer.
In Bauer et al. (2008), the authors have noted that a fixed fee rate will produce incentives for the policyholder to lapse the policy. The reason is that when the account value is higher, the guarantees (like put options) are out-of-the-money whereas the insured pays more fees. Thus, designing some proper fee charging method may reduce the possibility of surrendering the contract.

In Bernard et al. (2014a), under the simple Black-Scholes model, the authors proposed a state-dependent fee structure, under which fees are charged only if the account value is smaller than a pre-specified level.
Mathematically, if we denote by $F_t$ the value of the account at time $t$, then the total expenses charged by the insurer from time $0$ to $t$ are given by
\begin{equation}
\int_{0}^{t}\alpha_1F_{s}\rm{\bf{1}}_{\{F_{s} < B_1\}}ds,
\end{equation}
where $B_1$ is the pre-specified level. Similar fee charging strategy has also been considered in Zhou and Wu (2015) under the double exponential jump diffusion process. This kind of fee deducting approach has some advantages. For example, it can avoid effectively the problem that the policyholder lapses the policy due to the high fees when the guarantees are deep out-of-the-money. However, as the insurer cannot charge fees when the policyholder's account value becomes large, the fair fee rate $\alpha_1^*$ obtained is too high to be used in practice, see numerical results in section 4 of Bernard et al. (2014a), section 5 of Delong (2014) and section 4 of Zhou and Wu (2015).

From the investigation on optimal surrender for VAs (see Bernard et al. (2014b) for example), we know that the optimal surrender strategy under a fixed fee rate structure (i.e., $B_1=\infty $ in (1.1)) is a threshold strategy, i.e., the policyholder will lapse the policy when the value of the account exceeds a time-dependent barrier.
Although the optimal surrender strategy for $B_1 < \infty$ has not been investigated sufficiently, it is very likely that the policyholder will surrender the policy when the value of the account becomes large enough. This leads to the following problem: if the market rises continuously to the surrender barrier (just after the inception of the policy) and the insured  lapses the policy, then the insurer will have a few (or even no) fees income under the strategy (1.1).

In this paper, we extend the research in Bernard et al. (2014a) and Zhou and Wu (2015), and consider a multi-layer fee collecting method, under which the insurer charges expenses not only when the account value is lower than some pre-specified level but also when the account value exceeds some pre-specified level. That is, the total expenses deducted until time $t$ are given by
\begin{equation}
\int_{0}^{t}\left(\alpha_1F_{s}\rm{\bf{1}}_{\{F_{s} < B_1\}}+\alpha_2F_{s}\rm{\bf{1}}_{\{F_s \geq B_2\}}\right)ds,
\end{equation}
where $B_1 \leq B_2$ and usually $\alpha_2 < \alpha_1$. As formula (1.2) contains the term $\alpha_1F_s\rm{\bf{1}}_{\{F_s < B_1\}}$, the fee strategy (1.2) inherits some advantage of (1.1). Moreover, the second term $\alpha_2F_s\rm{\bf{1}}_{\{F_s \geq B_2\}}$ can alleviate the problem discussed in the last paragraph. This is because when $F_t$ exceeds $B_2$, this term can reduce the return of the policyholder's account by $\alpha_2$, which makes the account value exceed the surrender barrier with less probability. In addition, numerical examples in section 5 illustrate that the fair fee rates $\alpha_1^*$ and $\alpha_2^*$ are reasonable.

We should mention that Delong (2014) has considered a general state-dependent fee structure (which includes (1.2)) under a general L\'evy process. However, the fee rate obtained from the pricing principle (4.4) in his paper may yield arbitrage opportunities under some L\'evy process (e.g.,  a hyper-exponential jump diffusion process considered in this paper), see the discussion presented after formula (4.5) in Delong (2014). More importantly, we have derived formulas for the Laplace transforms of the distribution and the occupation times of a jump diffusion process with three-valued drift and hyper-exponential jumps. Our approach is remarkable and can be applied to more complicated fee structure (see Remark 2.1).

The reminder of this paper is organized as follows. Our model and an important preliminary result are given in section 2. In section 3, we derive formulas for the distribution of an important random variable, and then in section 4, we apply these formulas to price a variable annuity with guaranteed minimum maturity benefit. Finally, we give some numerical results in section 5 and draw conclusion in section 6.

\section{Details of the model and an important preliminary result}
Suppose that $X=(X_t)_{t \geq 0}$ is a hyper-exponential jump diffusion process, i.e.,
\begin{equation}
  X_t = X_0 + \mu t+\sigma W_t + \sum_{k=1}^{N_t}Z_k,
\end{equation}
where $\sigma > 0$, $\mu$ and $X_0$ are constants; $\{W_t; t\geq 0\}$ is a standard Brownian motion; $\{\sum_{k=1}^{N_t}Z_k; t\geq 0\}$, independent of $\{W_t; t\geq 0\}$, is a compound Poisson process; the intensity of the Poisson process $\{N_t; t\geq 0\}$ is given by $\lambda$ and the probability density function of $Z_1$ (denoted by $f_Z(z)$) is given by
\begin{equation}
f_{Z}(z) = \sum_{i=1}^{m}p_i\eta_i e^{-\eta_i z}\rm{\bf{1}}_{\{z > 0\}} + \sum_{j=1}^{n}q_j\vartheta_j e^{\vartheta_j z}\rm{\bf{1}}_{\{z < 0\}}.
\end{equation}
Here, in (2.2), $\textbf{1}_A$ is the indicator function of a set $A$; $p_i > 0$, $\eta_i > 0$ for all $i=1, \ldots, m$; $q_j > 0$, $\vartheta_j > 0$ for all $j=1, \ldots, n$; $\sum_{i=1}^{m}p_i+\sum_{j=1}^{n}q_j=1$. In addition, we assume that  $\eta_1<\eta_2<\cdots<\eta_m$ and $ \vartheta_1<\vartheta_2<\cdots<\vartheta_n$.

In this paper, we let $S_t$ be the time-t value of one unit of the reference fund underlying a variable annuity and assume that
\begin{equation}
S_t = S_0e^{X_t-X_0}.
\end{equation}
Under the multi-layer expense strategy (1.2), if we denote by $F_t$ the policyholder's account value at time $t$, then its dynamics are given by the following stochastic differential equation (SDE):
\begin{equation}
dF_t= F_{t-}\frac{d S_t}{S_{t-}}-\alpha_1 F_{t-} \rm{\bf{1}}_{\{F_{t-} < B_1\}}dt- \alpha_2 F_{t-} \rm{\bf{1}}_{\{F_{t-} \geq B_2\}}dt ,\ \ t>0,
\end{equation}
where $\alpha_i$ and $B_i$, $i=1,2$, are the deduction fee rate and the pre-specified level, respectively. Besides, the initial value $F_0$ is the single premium invested by the insured.

\begin{Remark}
Our approach can also be used to the following complicated fee structure:
\begin{equation}
\int_{0}^{t}\left(\alpha_1F_s\rm{\bf{1}}_{\{F_s < B_1\}}+\sum_{i=2}^{m}\alpha_iF_s\rm{\bf{1}}_{\{B_i \leq F_s < B_{i+1}\}}+
\alpha_{m+1}F_s\rm{\bf{1}}_{\{F_s \geq B_{m+2}\}}\right)ds,
\end{equation}
where $B_1 \leq B_2 \leq \cdots\leq B_{m+2}$. Note that formula (2.5) is similar to the so-called multi-layer dividend strategy, which is one of the motivations of this paper and has been studied by many papers (see Yang and Zhang (2009) for instance). This is the reason why our fee collecting method is called multi-layer expense strategy.
\end{Remark}

Set $U=(U_t)_{t \geq 0}$ be the unique strong solution to the following SDE
\begin{equation}
\begin{split}
&dU_t = dX_t - \alpha_1 \rm{\bf{1}}_{\{U_t < b_1\}}dt- \alpha_2 \rm{\bf{1}}_{\{U_t \geq b_2\}}dt,\ \  t>0,\\
&and  \ \ \ U_0 = X_0,
\end{split}
\end{equation}
where $b_1 = \ln\left(\frac{B_1}{F_0}\right)$ and $b_2 = \ln\left(\frac{B_2}{F_0}\right)$. Then, it follows from It$\hat{o}$'s formula that
\begin{equation}
F_t = F_0e^{U_t}, \ \ t\geq 0,
\end{equation}
if $U_0=0$. In addition, by using a similar arguments to that in Remark 3 of Kyprianou and Loeffen (2010), one can conclude that the process $U$ is a strong Markov process.

In this paper, we consider a VA with level GMMB, which has a payment $G(F_T)$ at its maturity $T$, where $G(\cdot)$ is a payoff function. An example of $G(\cdot)$ is that $G(x)=(K-x)_+$ for some $K > 0$. To price this VA, we need to compute the following expectation
\begin{equation}
 E^*\left[e^{-rT}G(F_T)\right],
\end{equation}
under an equivalent martingale measure $ P^*$ (i.e., $e^{-rt}e^{X_t}$ is martingale under $ P^*$), where $r$ denotes the continuously compounded constant risk-free rate. Here, similar to Zhou and Wu (2015), we use the Cram\'er-Esscher transform (proposed in Gerber and Shiu (1994)) to obtain the pricing martingale measure $P^*$. An attractive property of the Cram\'er-Esscher transform is that the process $X$, under the martingale measure $P^*$, is also a hyper-exponential jump diffusion process (see, e.g., Appendix A in Asmussen et al. (2004)).

Therefore, without loss of generality, we suppose that the process $X$, under the martingale measure $ P^*$, is given by (2.1) and (2.2). For the sake of brevity, in the following, we write simply $ P$ rather than $P^*$. Furthermore, we denote by $\mathbb P_x$ the law of $X$ such that $X_0=x$ and by $\mathbb E_x$ the corresponding expectation; and we write $\mathbb P$ and $\mathbb E$ when $X_0=0$ for simplicity. In addition, for $q>0$, we let $e(q)$ be an exponential random variable with expectation $1/q$, which is independent of the process $U$ under $\mathbb P_x$.

Of course, it is difficult to obtain the expression of (2.8) with $F_t$ given by (2.7). But, if we can obtain the distribution of $U_{e(q)}$ for some $q > 0$, then the Laplace transform of (2.8) can be computed as follows:
\begin{equation}
\int_{0}^{\infty}e^{-sT}\mathbb E\left[e^{-rT}G(F_T)\right]dT=\frac{1}{s+r}\mathbb E\left[G(F_{e(s+r)})\right].
\end{equation}
Moreover, the total time of deducting fees, i.e.,
\begin{equation}
\int_{0}^{T}\rm{\bf{1}}_{\{U_t < b_1\}}dt + \int_{0}^{T}\rm{\bf{1}}_{\{U_t \geq b_2\}}dt,
\end{equation}
can also be calculated from the distribution of $U_{e(q)}$ (see Remark 3.3). In short, deriving the expression for the probability distribution of $U_{e(q)}$ with $q > 0$ is a key task in this article.

\begin{Remark}
If we let $\tau$ represent the random variable denoting the time of death of the policyholder of a VA with level GMDB (whose payment at $\tau$ is $G_0(F_{\tau})$), then we can compute the price of this VA as following:
\begin{equation}
\mathbb E\left[e^{-r \tau }G_0(F_{\tau})\rm{\bf{1}}_{\{\tau < T\}}\right]= \int_{0}^{T}\mathbb E\left[e^{-r t }G_0(F_t)\right]g(t)dt,
\end{equation}
where $g(t)$ is the density function of $\tau$. As the term $\mathbb E\left[e^{-r t }G_0(F_t)\right]$ can be obtained from (2.9) via taking inverse Laplace transform, so (2.11) can be computed approximatively. However, this procedure involves massive numerical computations, especially when $T$ is large. Alternatively, we can use a similar idea to that in Remark 2.3 of Zhou and Wu (2015) to calculate (2.11) by using the results on the distribution of $U_{e(q)}$.
\end{Remark}

The following lemma, which characters the probability distribution function of $U_{e(q)}$, gives us some important boundary conditions (see (2.14), (2.17) and (2.20)). The proof of Lemma 2.1 is very similar to that of Theorem 2.1 in Zhou and Wu (2015), thus we omit the details.
\begin{Lemma}
(1) For given $y > b_2 > b_1$, consider a function $F_1(x)$ such that $F_1(x)$ is bounded and continuous on $\mathbb R$ and twice continuously differentiable on $\mathbb R$ except at $b_1$, $b_2$ and $y$. Assume $F_1(x)$
solves
\begin{equation}
\begin{split}
&\Gamma F_1(x)=qF_1(x),  \ \ x < y \ \ and \ \ x \neq b_1,\ b_2, \\
&\Gamma F_1(x)=qF_1(x)-q, \ \ x>y,
\end{split}
\end{equation}
where
\begin{equation}
\begin{split}
\Gamma F_1(x)=
&\frac{\sigma^2}{2}F_1^{\prime \prime}(x) + \left(\mu -\alpha_1\rm{\bf{1}}_{\{x < b_1\}}-\alpha_2\rm{\bf{1}}_{\{x \geq b_2\}}\right)F_1^{\prime}(x) \\
&+ \lambda \int_{-\infty}^{\infty}F_1(x + z)f_{Z}(z)dz- \lambda F_1(x).
\end{split}
\end{equation}
Moreover, assume that the derivative of $F_1(x)$ is continuous at $b_1$, $b_2$ and $y$, i.e.,
\begin{equation}
F_1^{\prime}(b_1-)=F_1^{\prime}(b_1+), \ \ F_1^{\prime}(b_2-)=F_1^{\prime}(b_2+) \ \ and \ \ F_1^{\prime}(y-)=F_1^{\prime}(y+).
\end{equation}
Then
\begin{equation}
F_1(x)=\mathbb P_x\left(U_{e(q)}> y\right).
\end{equation}

(2) For given $y < b_1 < b_2$, let $F_2(x)$ be bounded and continuous on $\mathbb R$ and twice continuously differentiable on $\mathbb R$ except at $b_1$, $b_2$ and $y$. If $F_2(x)$
solves
\begin{equation}
\begin{split}
&\Gamma F_2(x)=q F_2(x),  \ \ x > y \ \ and \ \ x \neq b_1, b_2, \\
&\Gamma F_2(x)=q F_2(x)-q, \ \ x < y,
\end{split}
\end{equation}
and satisfies
\begin{equation}
F_2^{\prime}(b_1-)=F_2^{\prime}(b_1+), \ \ F_2^{\prime}(b_2-)=F_2^{\prime}(b_2+) \ \ and \ \ F_2^{\prime}(y-)=F_2^{\prime}(y+),
\end{equation}
then
\begin{equation}
F_2(x)=\mathbb P_x\left(U_{e(q)}< y\right).
\end{equation}

(3) For given $b_1 < y < b_2$, let $F_3(x)$ be bounded and continuous on $\mathbb R$ and twice continuously differentiable on $\mathbb R$ except at $b_1$, $b_2$ and $y$. If $F_3(x)$
solves
\begin{equation}
\begin{split}
&\Gamma F_3(x)=q F_3(x),  \ \ x > b_1 \ \ and \ \ b_1< x < y,\\
&\Gamma F_3(x)=q F_3(x)- q, \ \ x > y \ \ and \ \ x \neq b_2,
\end{split}
\end{equation}
and satisfies
\begin{equation}
F_3^{\prime}(b_1-)=F_3^{\prime}(b_1+), \ \ F_3^{\prime}(b_2-)=F_3^{\prime}(b_2+) \ \ and \ \ F_3^{\prime}(y-)=F_3^{\prime}(y+),
\end{equation}
then
\begin{equation}
F_3(x)=\mathbb P_x\left(U_{e(q)} > y\right).
\end{equation}
\qed
\end{Lemma}
\begin{Remark}
Although Lemma 2.1 holds, it is not easy to obtain the distribution of $U_{e(q)}$ by solving (2.12), (2.16) and (2.19) directly. In next section, we derive the expression of the probability distribution of $U_{e(q)}$ by using an another approach rather than solving the equations in Lemma 2.1.
\end{Remark}

\section{The distribution of $U_{e(q)}$}
In this section, for the process $U$ determined by (2.1), (2.2) and (2.6), we first consider the case that $b_1 < b_2$ and derive formulas for $\mathbb P_x\left(U_{e(q)} > y\right)$ with $y > b_2$, $\mathbb P_x\left(U_{e(q)} < y\right)$ with $y < b_1$ and $\mathbb P_x\left(U_{e(q)} > y\right)$ with $b_1 < y < b_2$. After that the corresponding results for the case of $b_1=b_2$ are obtained by letting $b_2 \downarrow b_1$. Before presenting the results, we introduce some notation to end this paragraph. The L\'evy exponent of the process $X$ in (2.1) is given by
\begin{equation}
\psi(z):= \ln\left(\mathbb E\left[e^{z X_1}\right]\right)=\frac{\sigma^2}{2}z^2 + \mu z +\lambda\left(\sum_{i=1}^{m}\frac{p_i\eta_i}{\eta_i-z}+
\sum_{j=1}^{n}\frac{q_j\vartheta_j}{\vartheta_j+z}-1\right).
\end{equation}
For given $q > 0$,  the equation $\psi(z) = q$ has exactly $(m+n+2)$ real roots (denoted by $\beta_{1}$, $\beta_{2}$, $\ldots$,
 $\beta_{m+1}$, $-\gamma_{1}$, $-\gamma_{2}$, $\ldots$, $-\gamma_{n+1}$), which satisfy (see  Lemma 2.1 in Cai (2009) for the proof)
 \begin{equation}
 \begin{split}
&0 < \beta_{1} < \eta_1 < \beta_{2} <\cdots < \eta_m < \beta_{m+1}< \infty,\\
&0< \gamma_{1} < \vartheta_1 < \gamma_{2} < \cdots < \vartheta_n < \gamma_{n+1}< \infty.
\end{split}
\end{equation}
For the following two equations:
\begin{equation}
\tilde{\psi}(z):=\psi(z)-\alpha_1 z = q \ \ and \ \ \hat{\psi}(z):=\psi(z)-\alpha_2 z = q,
\end{equation}
the corresponding roots are denoted respectively by $\tilde{\beta}_{1}$, $\tilde{\beta}_{2}$, $\ldots$, $\tilde{\beta}_{m+1}$, $-\tilde{\gamma}_{1}$, $-\tilde{\gamma}_{2}$, $\ldots$, $-\tilde{\gamma}_{n+1}$ and $\hat{\beta}_{1}$, $\hat{\beta}_{2}$, $\ldots$,
$\hat{\beta}_{m+1}$, $-\hat{\gamma}_{1}$, $-\hat{\gamma}_{2}$, $\ldots$, $-\hat{\gamma}_{n+1}$.

The following three theorems give the expressions for $\mathbb P_x\left(U_{e(q)} > y\right)$ with $y > b_2$, $\mathbb P_x\left(U_{e(q)} < y\right)$ with $y < b_1$ and $\mathbb P_x\left(U_{e(q)} > y\right)$ with $b_1 < y < b_2$, respectively.
We give the details of the proof of Theorem 3.1 in Appendix B and omit the proofs of Theorems 3.2 and 3.3 as they are similar.

\begin{Theorem}
For given $b_1 < b_2$, $\alpha_1$ and $\alpha_2$ in (2.6), the expression of $\mathbb P_x\left(U_{e(q)} > y\right)$ for $y > b_2$ is given as follows.
\begin{equation}
\begin{split}
&\mathbb P_x\left(U_{e(q)}>y\right)=
\begin{small}
\left\{
\begin{aligned}
& \sum_{i=1}^{m+1}E_ie^{\tilde{\beta}_i(x-b_1)}, & x\leq b_1,\\
& \sum_{i=1}^{m+1}F_ie^{\beta_i(x-b_2)}+\sum_{j=1}^{n+1}G_je^{\gamma_j(b_1-x)}, & b_1 \leq  x \leq b_2, \\
& \sum_{i=1}^{m+1}H_ie^{\hat{\beta}_i(x-y)}+\sum_{j=1}^{n+1}M_je^{\hat{\gamma}_j(b_2-x)}, & b_2 \leq x \leq y,\\
&1+\sum_{j=1}^{n+1}N_je^{\hat{\gamma}_j(y-x)}+\sum_{j=1}^{n+1}M_je^{\hat{\gamma}_j(b_2-x)}, & x \geq y,
\end{aligned}
\right.
\end{small}
\end{split}
\end{equation}
where
\begin{small}
\begin{equation}
\begin{aligned}
&H_i=\frac{\prod_{k=1}^{m}(\eta_k-\hat{\beta}_i)\prod_{k=1, k\neq i}^{m+1}\hat{\beta}_k
\prod_{k=1}^{m}(\hat{\beta}_i+\vartheta_k)\prod_{k=1}^{m+1}\hat{\gamma}_k}
{\prod_{k=1}^{m}\eta_k \prod_{k=1, k\neq i}^{m+1}(\hat{\beta}_k-\hat{\beta}_i)\prod_{k=1}^{m}\vartheta_k
\prod_{k=1}^{m+1}(\hat{\beta}_i+\hat{\gamma}_k)}, & 1\leq i\leq m+1, \\
&N_j=-\frac{\prod_{k=1}^{n}(\vartheta_k-\hat{\gamma}_j)
\prod_{k=1, k\neq j}^{n+1}\hat{\gamma}_k
\prod_{k=1}^{m}(\hat{\gamma}_j+\eta_k)
\prod_{k=1}^{m+1}\hat{\beta}_k
}{\prod_{k=1}^{n}\vartheta_k \prod_{k=1, k\neq j}^{n+1}(\hat{\gamma}_k-\hat{\gamma}_j)
\prod_{k=1}^{m}\eta_k \prod_{k=1}^{m+1}(\hat{\gamma}_j+\hat{\beta}_k)},& 1\leq j\leq n+1.
\end{aligned}
\end{equation}
\end{small}
The other constants in (3.4), $E_1, \ldots, E_{m+1}$, $F_1, \ldots, F_{m+1}$, $G_1, \ldots, G_{n+1}$ and $M_1, \ldots, M_{n+1}$, are determined by
\begin{equation}
\left(E_1,\ldots,E_{m+1},F_1,\ldots,F_{m+1},G_1,\ldots,G_{n+1},M_1, \ldots,M_{n+1}\right)Q_1=h,
\end{equation}
with $Q_1$ given by (A.1) in Appendix A and
\begin{equation}
h=\left(0,\ldots,0,h_{m+n+3},\ldots,h_{2m+2n+4}\right),
\end{equation}
where
\begin{equation}
\begin{split}
&h_{m+n+3}=\sum_{i=1}^{m+1}H_ie^{\hat{\beta}_i(b_2-y)}, \ \ h_{m+n+4}=\sum_{i=1}^{m+1}H_i\hat{\beta}_i e^{\hat{\beta}_i(b_2-y)},\\
&h_{m+n+4+k}=\sum_{i=1}^{m+1}\frac{H_i\vartheta_k}{
\vartheta_k+\hat{\beta}_i}e^{\hat{\beta}_i(b_2-y)},\ \ k=1,\ldots, n,\\
&h_{m+2n+4+k}=\sum_{i=1}^{m+1}\frac{H_i\eta_k}{\eta_k-\hat{\beta}_i}e^{\hat{\beta}_i(b_2-y)}, \ \ k=1,\ldots, m.
\end{split}
\end{equation}
\end{Theorem}

\begin{Remark}
Intuitively, the columns of the matrix $Q_1$ in (A.1) are linearly independent. In other words, the matrix $Q_1$ is nonsingular. At present, we cannot find an easy approach to show this fact. However, a large number of numerical calculations, including those in section 5, confirm that $Q_1$ is an invertible matrix.
\end{Remark}

\begin{Theorem}
For given $b_1 < b_2$, $\alpha_1$ and $\alpha_2$ in (2.6), the expression of $\mathbb P_x\left(U_{e(q)} < y\right)$ for $y < b_1$ is given as follows.

\begin{equation}
\mathbb P_x\left(U_{e(q)} < y\right)=
\begin{small}
\left\{
\begin{aligned}
&1 + \sum_{i=1}^{m+1}\tilde{E}_ie^{\tilde{\beta}_i(x-y)}+
\sum_{i=1}^{m+1}\tilde{F}_i e^{\tilde{\beta}_i(x-b_1)}, & x\leq y,\\
& \sum_{i=1}^{m+1}\tilde{F}_ie^{\tilde{\beta}_i(x-b_1)}
+\sum_{j=1}^{n+1}\tilde{G}_je^{\tilde{\gamma}_j(y-x)}, & y \leq  x \leq b_1, \\
& \sum_{j=1}^{n+1}\tilde{H}_je^{\gamma_j(b_1-x)}+\sum_{i=1}^{m+1}\tilde{M}_i e^{\beta_i(x-b_2)}, & b_1 \leq x \leq b_2,\\
&\sum_{j=1}^{n+1}\tilde{N}_je^{\hat{\gamma}_j(b_2-x)}, & x \geq b_2,
\end{aligned}
\right.
\end{small}
\end{equation}
where
\begin{small}
\begin{equation}
\begin{split}
&\tilde{G}_j=\frac{\prod_{k=1}^{n}(\vartheta_k-\tilde{\gamma}_j)
\prod_{k=1, k\neq j}^{n+1}\tilde{\gamma}_k\prod_{k=1}^{m}(\tilde{\gamma}_j+\eta_k)
\prod_{k=1}^{m+1}\tilde{\beta}_k}{\prod_{k=1}^{n}\vartheta_k
\prod_{k=1, k\neq j}^{n+1}(\tilde{\gamma}_k-\tilde{\gamma}_j)
\prod_{k=1}^{m}\eta_k
\prod_{k=1}^{m+1}(\tilde{\gamma}_j+\tilde{\beta}_k)}, \ \ 1\leq j \leq  n+1,\\
&\tilde{E}_i=-\frac{\prod_{k=1}^{m}(\eta_k-\tilde{\beta}_i)
\prod_{k=1, k\neq i}^{m+1}\tilde{\beta}_k
\prod_{k=1}^{n}(\tilde{\beta}_i+\vartheta_k)
\prod_{k=1}^{n+1}\tilde{\gamma}_k}{
\prod_{k=1}^{m}\eta_k
\prod_{k=1, k\neq i}^{m+1}(\tilde{\beta}_k-\tilde{\beta}_i)
\prod_{k=1}^{n}\vartheta_k
\prod_{k=1}^{n+1}(\tilde{\beta}_i+\tilde{\gamma}_k)}, \ \ 1 \leq i \leq m+1,
\end{split}
\end{equation}
\end{small}
and
\begin{equation}
\left(\tilde{F}_1,\ldots,\tilde{F}_{m+1},\tilde{M}_1,\ldots,\tilde{M}_{m+1},\tilde{H}_1,\ldots,\tilde{H}_{n+1}
,\tilde{N}_1, \ldots, \tilde{N}_{n+1}\right)Q_1=\tilde{h},
\end{equation}
with
\begin{equation}
\tilde{h}=\left(\tilde{h}_1, \ldots, \tilde{h}_{2+n+m},0,\ldots,0\right),
\end{equation}
and
\begin{equation}
\begin{split}
&\tilde{h}_{1}=-\sum_{j=1}^{n+1}\tilde{G}_je^{\tilde{\gamma}_j(y-b_1)}, \ \ \tilde{h}_{2}=\sum_{j=1}^{n+1}\tilde{G}_j\tilde{\gamma}_j e^{\tilde{\gamma}_j(y-b_1)}, \\
& \tilde{h}_{2+k}=-\sum_{j=1}^{n+1}\frac{\tilde{G}_j \vartheta_k}{\vartheta_k-\tilde{\gamma}_j}
e^{\tilde{\gamma}_j(y-b_1)}, \ \ \ \ \ k=1,\ldots, n, \\
&\tilde{h}_{2+n+k}=-\sum_{j=1}^{n+1}\frac{\tilde{G}_j\eta_k}{
\eta_k+\tilde{\gamma}_j}e^
{\tilde{\gamma}_j(y-b_1)},\ \ k=1,\ldots, m.
\end{split}
\end{equation}
\end{Theorem}

\begin{Theorem}
For given $b_1 < y < b_2$, $\alpha_1$ and $\alpha_2$ in (2.6), we have the following results.
\begin{equation}
\mathbb P_x\left(U_{e(q)} > y\right)=
\begin{small}
\left\{\begin{aligned}
&\sum_{i=1}^{m+1}\hat{E}_ie^{\tilde{\beta}_i(x-b_1)}, & x \leq b_1,\\
&\sum_{i=1}^{m+1}\left(\hat{U}_i+\hat{H}_ie^{\beta_i(y-b_2)}\right)e^{\beta_i(x-y)}
+\sum_{j=1}^{n+1}\hat{G}_je^{\gamma_j(b_1-x)}, & b_1 \leq  x \leq y, \\
&1+\sum_{i=1}^{m+1}\hat{H}_i e^{\beta_i (x-b_2)}+\sum_{j=1}^{n+1}\left(\hat{V}_j+\hat{G}_je^{\gamma_j(b_1-y)}\right) e^{\gamma_j(y-x)}
, & y \leq x \leq b_2,\\
&1+\sum_{j=1}^{n+1}\hat{N}_j e^{\hat{\gamma}_j(b_2-x)}, & x \geq b_2,
\end{aligned}
\right.
\end{small}
\end{equation}
where
\begin{small}
\begin{equation}
\begin{aligned}
&\hat{V}_j=-\frac{\prod_{k=1}^{n}(\vartheta_k-\gamma_j)
\prod_{k=1, k\neq j}^{n+1}\gamma_k\prod_{k=1}^{m}(\gamma_j+\eta_k)
\prod_{k=1}^{m+1}\beta_k}{\prod_{k=1}^{n}\vartheta_k
\prod_{k=1, k\neq j}^{n+1}(\gamma_k-\gamma_j)
\prod_{k=1}^{m}\eta_k
\prod_{k=1}^{m+1}(\gamma_j+\beta_k)}, & 1\leq j\leq n+1,\\
&\hat{U}_i=\frac{\prod_{k=1}^{m}(\eta_k-\beta_i)
\prod_{k=1, k\neq i}^{m+1}\beta_k
\prod_{k=1}^{n}(\beta_i+\vartheta_k)
\prod_{k=1}^{n+1}\gamma_k}{
\prod_{k=1}^{m}\eta_k
\prod_{k=1, k\neq i}^{m+1}(\beta_k-\beta_i)
\prod_{k=1}^{n}\vartheta_k
\prod_{k=1}^{n+1}(\beta_i+\gamma_k)}, & 1 \leq i \leq m+1,
\end{aligned}
\end{equation}
\end{small}
and
\begin{equation}
\left(\hat{E}_1,\ldots,\hat{E}_{m+1},\hat{H}_1,\ldots,\hat{H}_{m+1},\hat{G}_1,\ldots,\hat{G}_{n+1}
,\hat{N}_1, \ldots, \hat{N}_{n+1}\right)Q_1=\hat{h}.
\end{equation}
Here, the vector $\hat{h}$ in (3.16) is given by
\begin{equation}
\hat{h}=\left(\hat{h}_1,\ldots, \hat{h}_{2m+2n+4}\right),
\end{equation}
where
\begin{equation}
\begin{split}
&\hat{h}_{1}=\sum_{i=1}^{m+1}\hat{U}_ie^{\beta_i(b_1-y)}, \ \ \hat{h}_{2}=\sum_{i=1}^{m+1}\hat{U}_i\beta_ie^{\beta_i(b_1-y)},\\
&\hat{h}_{2+k}=\sum_{i=1}^{m+1}\frac{\hat{U}_i\vartheta_k}{\vartheta_k+\beta_i}
e^{\beta_i(b_1-y)}, \ \ k=1,\ldots, n, \\
&\hat{h}_{2+n+k}=\sum_{i=1}^{m+1}\frac{\hat{U}_i\eta_k}{\eta_k-\beta_i}
e^{\beta_i(b_1-y)},\ \ k=1,\ldots,m,
\end{split}
\end{equation}
and
\begin{equation}
\begin{split}
&\hat{h}_{m+n+3}=-\sum_{j=1}^{n+1}\hat{V}_j e^{\gamma_j(y-b_2)}, \ \ \hat{h}_{m+n+4}=\sum_{j=1}^{n+1}\hat{V}_j\gamma_je^{\gamma_j(y-b_2)},\\
&\hat{h}_{m+n+4+k}=-\sum_{j=1}^{n+1}\frac{\hat{V}_j\vartheta_k}{\vartheta_k-\gamma_j}
e^{\gamma_j(y-b_2)}, \ \ k=1,\ldots, n,\\
&\tilde{h}_{m+2n+4+k}=-\sum_{j=1}^{n+1}\frac{\hat{V}_j\eta_k}{\eta_k+\gamma_j}
e^{\gamma_j(y-b_2)},\ \ k=1,\ldots,m.
\end{split}
\end{equation}
\end{Theorem}

\begin{Remark}
From the derivation of Theorem 3.3, we have
\begin{equation}
\begin{split}
&\sum_{i=1}^{m+1}\hat{U}_i-\sum_{j=1}^{n+1}\hat{V}_j-1=0, \\ &\sum_{i=1}^{m+1}\hat{U}_i\beta_i+\sum_{j=1}^{n+1}\hat{V}_j\gamma_j=0,\\
&\sum_{i=1}^{m+1}\frac{\hat{U}_i\vartheta_k}{\vartheta_k+\beta_i}-\sum_{j=1}^{n+1}\frac{\hat{V}_j\vartheta_k}
{\vartheta_k-\gamma_j}-1=0,\ \ k=1, 2,\ldots, n, \\ &\sum_{i=1}^{m+1}\frac{\hat{U}_i\eta_k}{\eta_k-\beta_i}-\sum_{j=1}^{n+1}\frac{\hat{V}_j\eta_k}
{\eta_k+\gamma_j}-1=0, \ \ k=1,2,\ldots, m.
\end{split}
\end{equation}
We remark that formula (3.15) is obtained easily via solving (3.20) and that the equations in (3.20) will be used in proof of the following Corollary 3.1.
\end{Remark}

In order to obtain the distribution of $U_{e(q)}$, the above three theorems are not enough as they do not give the values of the two probabilities: $\mathbb P_x\left(U_{e(q)}=b_2\right)$ and $\mathbb P_x\left(U_{e(q)}=b_1\right)$. In Corollary 3.1 (whose proof is given in Appendix C), we draw the conclusion that both of them are equal to zero. We remark that this result is not surprising because we have assumed that $\sigma > 0$ in this paper.
\begin{Corollary}
For given $b_1 < b_2$, $\alpha_1$ and $\alpha_2$ in (2.6), we have
\begin{equation}
\mathbb P_x\left(U_{e(q)}=b_2\right)=\mathbb P_x\left(U_{e(q)}=b_1\right)=0.
\end{equation}
\end{Corollary}

\begin{Remark}
For $y \in \mathbb R$, applying integration by part yields
\begin{equation}
\begin{split}
&\int_{0}^{\infty}e^{-qT}\mathbb E_x\left[\int_{0}^{T}\rm{\bf{1}}_{\{U_t \geq  y\}}dt\right]dT=\frac{\mathbb P_x\left(U_{e(q)} \geq y\right)}{q^2}, \\
&\int_{0}^{\infty}e^{-qT}\mathbb E_x\left[\int_{0}^{T}\rm{\bf{1}}_{\{U_t <  y\}}dt\right]dT=\frac{\mathbb P_x\left(U_{e(q)} < y\right)}{q^2},
\end{split}
\end{equation}
which can be computed from Theorems 3.1, 3.2, 3.3 and Corollary 3.1. Particularly, note that the sum of $\mathbb E_x\left[\int_{0}^{T}\rm{\bf{1}}_{\{U_t \geq  b_2\}}dt\right]$ and $\mathbb E_x\left[\int_{0}^{T}\rm{\bf{1}}_{\{U_t <  b_1\}}dt\right]$ is the total time of deducting fees for a VA with GMMB rider under the multi-layer expense strategy (1.2).
\end{Remark}

For fixed $b_1$, if we let $b_2 \downarrow b_1$ in Theorems 3.1 and 3.2, then we can derive formulas for the distribution function of $U_{e(q)}$ with $b_1 = b_2$ in the following Corollary 3.2. In Appendix C, we give  the details of deriving Corollary 3.2.
\begin{Corollary}
For given $b_1 = b_2$, $\alpha_1$ and $\alpha_2$ in (2.6), we have the following results.

(i) For $y > b_1$,
\begin{equation}
\begin{split}
&\mathbb P_x\left(U_{e(q)} > y\right)=
\begin{small}
\left\{
\begin{aligned}
& \sum_{i=1}^{m+1}E^1_ie^{\tilde{\beta}_i(x-b_1)}, & x\leq b_1,\\
& \sum_{i=1}^{m+1}H_ie^{\hat{\beta}_i(x-y)}+\sum_{j=1}^{n+1}M^1_je^{\hat{\gamma}_j(b_1-x)}, & b_1 \leq x \leq y,\\
&1+\sum_{j=1}^{n+1}N_je^{\hat{\gamma}_j(y-x)}+\sum_{j=1}^{n+1}M^1_je^{\hat{\gamma}_j(b_1-x)}, & x \geq y,
\end{aligned}
\right.
\end{small}
\end{split}
\end{equation}
where $H_i$ and $N_j$ are given by (3.5) and
\begin{equation}
\begin{split}
&M^1_{j}
=\frac{\prod_{k=1}^{m}(\hat{\gamma}_{j}+\eta_k)\prod_{k=1}^{n}
(\vartheta_k-\hat{\gamma}_{j})\prod_{k=1}^{m+1}\hat{\beta}_{k}\prod_{k=1}^{n+1}\hat{\gamma}_{k}}
{\prod_{k=1}^{m}\eta_k
\prod_{k=1}^{n}\vartheta_k\prod_{k=1}^{m+1}(\hat{\gamma}_{j}+\tilde{\beta}_{k})\prod_{k=1,k\neq j}^{n+1}
(\hat{\gamma}_{k}-\hat{\gamma}_{j})}\\
&\ \ \ \ \ \ \ \ \ \ \ \ \times\sum_{i=1}^{m+1}\frac{\prod_{k=1}^{m+1}(\hat{\beta}_{i}-\tilde{\beta}_{k})
e^{\hat{\beta}_{i}(b_1-y)}}
{\hat{\beta}_{i}(\hat{\beta}_{i}
+\hat{\gamma}_{j})\prod_{k=1,k\neq i}^{m+1}(\hat{\beta}_{i}-\hat{\beta}_{k})}, \ \ j=1, \ldots, n+1,
\end{split}
\end{equation}
and
\begin{equation}
\begin{split}
&E^1_i
=\frac{\prod_{k=1}^{n+1}\hat{\beta}_{k}\prod_{k=1}^{n+1}\hat{\gamma}_{k}
\prod_{k=1}^{m}(\tilde{\beta}_{i}-\eta_k)\prod_{k=1}^{n}(\tilde{\beta}_{i}+\vartheta_k)}
{\prod_{k=1}^{m}\eta_k
\prod_{k=1}^{n}\vartheta_k\prod_{k=1,k\neq i}^{m+1}(\tilde{\beta}_{i}-\tilde{\beta}_{k})\prod_{k=1}^{n+1}(\tilde{\beta}_{i}+\hat{\gamma}_{k})}\\
&\ \ \ \ \ \times \sum_{j=1}^{m+1}\frac{\prod_{k=1,k\neq i}^{m+1}(
\hat{\beta}_{j}-\tilde{\beta}_{k})}{\hat{\beta}_{j}\prod_{k=1,k\neq j}^{m+1}(\hat{\beta}_{j}-\hat{\beta}_{k})}e^{\hat{\beta}_{j}(b_1-y)}, \ \ i=1, \ldots, m+1.
\end{split}
\end{equation}

(ii) For $y < b_1$,
\begin{equation}
\mathbb P_x\left(U_{e(q)} < y\right)=
\begin{small}
\left\{
\begin{aligned}
&1 + \sum_{i=1}^{m+1}\tilde{E}_ie^{\tilde{\beta}_i(x-y)}+
\sum_{i=1}^{m+1}\tilde{F}^1_i e^{\tilde{\beta}_i(x-b_1)}, & x\leq y,\\
& \sum_{i=1}^{m+1}\tilde{F}^1_ie^{\tilde{\beta}_i(x-b_1)}
+\sum_{j=1}^{n+1}\tilde{G}_je^{\tilde{\gamma}_j(y-x)}, & y \leq  x \leq b_1, \\
&\sum_{j=1}^{n+1}\tilde{N}^1_je^{\hat{\gamma}_j(b_1-x)}, & x \geq b_1,
\end{aligned}
\right.
\end{small}
\end{equation}
where $\tilde{E}_i$ and $\tilde{G}_j$ are given by (3.10) and
\begin{equation}
\begin{split}
&\tilde{N}^1_{j}
=
\frac{\prod_{k=1}^{m+1}\tilde{\beta}_{k}\prod_{k=1}^{n+1}
\tilde{\gamma}_{k}\prod_{k=1}^{m}(\hat{\gamma}_{j}+\eta_k)\prod_{k=1}^{n}
(\vartheta_k-\hat{\gamma}_{j})}{\prod_{k=1}^{m}\eta_k\prod_{k=1}^{n}\vartheta_k
\prod_{k=1}^{m+1}(\hat{\gamma}_{j}+\tilde{\beta}_{k})\prod_{k=1,k\neq j}^{n+1}
(\hat{\gamma}_{k}-\hat{\gamma}_{j})}\\
&\ \ \ \ \ \ \ \ \ \ \ \times\sum_{i=1}^{n+1}\frac{\prod_{k=1,k\neq j}^{n+1}(\hat{\gamma}_{k}-\tilde{\gamma}_{i})}{\tilde{\gamma}_{i}\prod_{k=1,k \neq i}^{n+1}
(\tilde{\gamma}_{k}-\tilde{\gamma}_{i})}
e^{\tilde{\gamma}_{i}(y-b_1)}, \ \ j=1,\ldots, n+1,
\end{split}
\end{equation}

\begin{equation}
\begin{split}
&\tilde{F}^1_i
=\frac{\prod_{k=1}^{m+1}\tilde{\beta}_{k}\prod_{k=1}^{n+1}
\tilde{\gamma}_{k}\prod_{k=1}^{m}(\tilde{\beta}_{i}-\eta_k)\prod_{k=1}^{n}
(\tilde{\beta}_{i}+\vartheta_k)}{\prod_{k=1}^{m}\eta_k\prod_{k=1}^{n}\vartheta_k \prod_{k=1,k\neq i}^{m+1}(\tilde{\beta}_{i}-\tilde{\beta}_{k})\prod_{k=1}^{n+1}
(\tilde{\beta}_{i}+\hat{\gamma}_{k})}\\
&\ \ \ \ \ \ \ \times\sum_{j=1}^{n+1}\frac{\prod_{k=1}^{n+1}(\hat{\gamma}_{k}-\tilde{\gamma}_{j})e^{\tilde{\gamma}_{j}(y-b_1)}}
{-\tilde{\gamma}_{j}(\tilde{\beta}_{i}+\tilde{\gamma}_{j})\prod_{k=1,k\neq j}^{n+1}
(\tilde{\gamma}_{k}-\tilde{\gamma}_{j})}, \ \ i=1,\ldots, m+1.
\end{split}
\end{equation}

(iii) For $x \in \mathbb R$,
\begin{equation}
\mathbb P_x\left(U_{e(q)}=b_1\right)=0.
\end{equation}
\end{Corollary}

\begin{Remark}
Corollary 3.2 extends the results in Theorem 3.1 of Zhou and Wu (2015) from the double exponential jump diffusion process to the hyper-exponential jump diffusion process.
\end{Remark}

\section{Evaluating variable annuities under the multi-layer expense strategy}
In this section, we apply the results in section 3 to evaluate a variable annuity with the multi-layer expense strategy. For the sake of simplicity, we only investigate the case that $G(x)$ in (2.8) is given by $G(x)=(K-x)_+$ for some guaranteed level $K$.

For given $B_1 \leq B_2$ in (2.4), the fair fee rates $\alpha_1^*$ and $\alpha_2^*$ are computed such that
the initial premium equals the expected value of the discounted payoff, i.e.,
\begin{equation}
F_0=\mathbb E\left[e^{-r T}\max\{F_{T},K\}\right]=\mathbb E\left[e^{-r T}F_{T}\right]+\mathbb E\left[e^{-r T}(K-F_{T})_+\right].
\end{equation}
Besides, with the fair fee rates $\alpha_1^*$ and $\alpha_2^*$, we want to calculate the total fees that will be deducted, i.e.,
\begin{equation}
\int_{0}^{T}e^{-r t}\left(\alpha_1^* F_t\textbf{1}_{\{F_t < B_1\}}+\alpha_2^*F_t\textbf{1}_{\{F_t \geq B_2\}}\right)
dt,
\end{equation}
whose expectation equals (by It$\hat{o}$'s formula)
\begin{equation}
\begin{split}
&\mathbb E\left[\int_{0}^{T}e^{-rt}\left(\alpha_1^* F_t\textbf{1}_{\{F_t < B_1\}}+\alpha_2^*F_t\textbf{1}_{\{F_t \geq B_2\}}\right)dt\right]\\
=&\mathbb E\left[e^{- r T}(F_0e^{X_{T}} - F_T)\right] = F_0 - \mathbb E\left[e^{- r T}F_T\right],
\end{split}
\end{equation}
where the second equality follows from the fact that $e^{-rt}e^{X_t}$ is a martingale.

Formulas (4.1) and (4.3) imply that we need to calculate the two expectations: $\mathbb E\left[e^{- r T}F_T\right]$ and  $\mathbb E\left[e^{-r T}(K-F_{T})_+\right]$ with given $B_1$, $B_2$, $\alpha_1$ and $\alpha_2$. In the following, we will derive the Laplace transforms of these two expectations with respect to $T$. In order to avoid introducing more notation, we only consider the case that $K=F_0$, $B_1=F_0(=K)$ and $B_2 > F_0$ (other cases can be discussed in a similar way).

\begin{Remark}
The case that $K$ equals $F_0$ is known as the "return-of-premium" guarantee. In addition, the situation of $B_1=F_0=K$ is interesting, as it means that the insurer deducts  fees for the embedded guarantee if the guarantee (like a put option) is in-the-money.
\end{Remark}

$\bullet$ Results on $\int_{0}^{\infty}e^{-sT}\mathbb E\left[e^{- r T}F_T\right]dT=\frac{1}{q}\mathbb E\left[F_{e(q)}\right]$ with $q=r+s$.

For $b_1=\ln(B_1/F_0)=0$ and $b_2=\ln(B_2/F_0) > 0$, we have
\begin{equation}
\begin{split}
\mathbb E\left[F_{e(q)}\right]
&=\int_{-\infty}^{\infty}F_0 e^y\mathbb P\left(U_{e(q)}\in dy \right)
=F_0-\int_{-\infty}^{0}F_0 e^y \mathbb P(U_{e(q)}< y)dy\\
&\ \ +\int_{0}^{b_2}F_0 e^y \mathbb P\left(U_{e(q)}> y \right)dy+\int_{b_2}^{\infty}F_0 e^y\mathbb P\left(U_{e(q)}> y \right)dy,
\end{split}
\end{equation}
where we have used the fact that $\lim_{y\uparrow \infty}e^y\mathbb P\left(
U_{e(q)}> y\right)=0$ (which is due to that $\mathbb E\left[e^{U_{e(q)}}\right] < \mathbb E\left[e^{X_{e(q)}}\right] < \infty$ for $q > r$) in the second equality.

 It follows from Theorem 3.1 that
\begin{equation}
\begin{split}
&\int_{b_2}^{\infty} e^y\mathbb P\left(U_{e(q)}> y \right)dy=\int_{b_2}^{\infty}e^y
\sum_{i=1}^{m+1}E_i dy\\
&=\left(\int_{b_2}^{\infty}e^y
h_1 dy, \ldots, \int_{b_2}^{\infty}e^y
h_{2m+2n+4}dy\right)Q_1^{-1}w^{T},
\end{split}
\end{equation}
where
\begin{equation}
w=(\underbrace{1, \ldots, 1}_{m+1}, \underbrace{0, \cdots, 0}_{m+2n+3}).
\end{equation}
Similarly, from Theorems 3.2 and 3.3, we have
\begin{equation}
\begin{split}
&\int_{-\infty}^{0}e^y \mathbb P(U_{e(q)}< y)dy=\int_{-\infty}^{0}e^y \left(\sum_{i=1}^{m+1}\tilde{F}_i+\sum_{j=1}^{n+1}\tilde{G}_je^{\tilde{\gamma}_jy}\right)dy\\
&=\sum_{j=1}^{n+1}\frac{\tilde{G}_j}{1+\tilde{\gamma}_j}+\left(\int_{-\infty}^{0}e^y \tilde{h}_1dy, \ldots,
\int_{-\infty}^{0}e^y \tilde{h}_{2m+2n+4}dy\right)Q_1^{-1}w^{T},
\end{split}
\end{equation}
and
\begin{equation}
\begin{split}
&\int_{0}^{b_2} e^y \mathbb P\left(U_{e(q)}> y \right)dy=\int_{0}^{b_2} e^y\sum_{i=1}^{m+1}\hat{E}_idy\\
&=\left(\int_{0}^{b_2} e^y \hat{h}_1 dy, \ldots, \int_{0}^{b_2} e^y \hat{h}_{2m+2n+4} dy \right)Q_1^{-1}w^{T}.
\end{split}
\end{equation}

Therefore, we obtain that
\begin{equation}
\frac{1}{q}\mathbb E\left[F_{e(q)}\right]=\frac{F_0}{q}+\frac{1}{q}F_0\left(v_1, \ldots, v_{2m+2n+4}\right)Q_1^{-1}w^{T}-\frac{1}{q}\sum_{j=1}^{n+1}\frac{F_0\tilde{G}_j}
{1+\tilde{\gamma}_j},
\end{equation}
where
\begin{small}
\begin{equation}
\begin{split}
&v_1=\sum_{i=1}^{m+1}\frac{\hat{U}_i}{\beta_i-1}\left(1-e^{(1-\beta_i)b_2}\right)+\sum_{j=1}^{n+1}
\frac{\tilde{G}_j}{1+\tilde{\gamma}_j},\\
&v_2=\sum_{i=1}^{m+1}\frac{\hat{U}_i\beta_i}{\beta_i-1}\left(1-e^{(1-\beta_i)b_2}\right)-\sum_{j=1}^{n+1}
\frac{\tilde{G}_j \tilde{\gamma}_j}{1+\tilde{\gamma}_j},\\
&v_{2+k}=\sum_{i=1}^{m+1}\frac{\hat{U}_i\vartheta_k\left(1-e^{(1-\beta_i)b_2}\right)}{(\vartheta_k+\beta_i)(\beta_i-1)}
+\sum_{j=1}^{n+1}
\frac{\tilde{G}_j\vartheta_k}{(\vartheta_k-\tilde{\gamma}_j)(1+\tilde{\gamma}_j)},\ \ k=1,\ldots, n,\\
&v_{2+n+k}=\sum_{i=1}^{m+1}\frac{\hat{U}_i\eta_k \left(1-e^{(1-\beta_i)b_2}\right)}{(\eta_k-\beta_i)(\beta_i-1)}
+\sum_{j=1}^{n+1}
\frac{\tilde{G}_j\eta_k}{(\eta_k+\tilde{\gamma}_j)(1+\tilde{\gamma}_j)},\ \ k=1,\ldots, m,
\end{split}
\end{equation}
\end{small}
and
\begin{small}
\begin{equation}
\begin{split}
&v_{m+n+3}=\sum_{j=1}^{n+1}\frac{\hat{V}_j}{\gamma_j+1}\left(e^{-\gamma_jb_2}-e^{b_2}\right)+\sum_{i=1}^{m+1}
\frac{\tilde{H}_ie^{b_2}}{\hat{\beta}_i-1},\\
&v_{m+n+4}=\sum_{j=1}^{n+1}\frac{\hat{V}_j\gamma_j}{\gamma_j+1}\left(e^{b_2}-e^{-\gamma_jb_2}\right)+\sum_{i=1}^{m+1}
\frac{H_i\hat{\beta}_i e^{b_2}}{\hat{\beta}_i-1},\\
&v_{m+n+4+k}=\sum_{j=1}^{n+1}\frac{\hat{V}_j\vartheta_k\left(e^{-\gamma_jb_2}-e^{b_2}\right)}{(\vartheta_k-\gamma_j)(\gamma_j+1)}
+\sum_{i=1}^{m+1}
\frac{H_i \vartheta_k e^{b_2}}{(\vartheta_k+\hat{\beta}_i)(\hat{\beta}_i-1)},\ \ 1 \leq k \leq n,\\
&v_{m+2n+4+k}=\sum_{j=1}^{n+1}\frac{\hat{V}_j\eta_k\left(e^{-\gamma_jb_2}-e^{b_2}\right)}{(\eta_k+\gamma_j)(\gamma_j+1)}
+
\sum_{i=1}^{m+1}
\frac{H_i \eta_k e^{b_2}}{(\eta_k-\hat{\beta}_i)(\hat{\beta}_i-1)},\ \ 1 \leq k \leq m.
\end{split}
\end{equation}
\end{small}

$\bullet$ Results on $\int_{0}^{\infty}e^{-sT}\mathbb E\left[e^{-r T}(F_0-F_{T})_+\right]dT=\frac{1}{q}
\mathbb E\left[(F_0-F_{e(q)})_+\right]$.

For $b_1=\ln(B_1/F_0)=0$ and $b_2=\ln(B_2/F_0) > 0$, we have
\begin{equation}
\begin{split}
&\frac{1}{q}\mathbb E\left[(F_0-F_{e(q)})_+\right]
=\frac{F_0}{q}\int_{-\infty}^{0}e^y\mathbb P\left(U_{e(q)}< y\right)dy\\
&=
\frac{1}{q}\sum_{j=1}^{n+1}\frac{F_0\tilde{G}_j}{1+\tilde{\gamma}_j}
+\frac{1}{q}F_0\left(\tilde{v}_1, \ldots, \tilde{v}_{2m+2n+4}\right)Q_1^{-1}w^{T},
\end{split}
\end{equation}
where
\begin{equation}
\begin{aligned}
&\tilde{v}_1=-\sum_{j=1}^{n+1}\frac{\tilde{G}_j}{1+\tilde{\gamma}_j},
\ \ \tilde{v}_2=\sum_{j=1}^{n+1}\frac{\tilde{G}_j\tilde{\gamma}_j}{1+\tilde{\gamma}_j},\\
&\tilde{v}_{2+k}=-\sum_{j=1}^{n+1}\frac{\tilde{G}_j\vartheta_k}{(\vartheta_k-\tilde{\gamma}_j)
(1+\tilde{\gamma}_j)}, \ \ \ \ k=1,2,\ldots, n,\\
&\tilde{v}_{2+n+k}=-\sum_{j=1}^{n+1}\frac{\tilde{G}_j\eta_k}{(\eta_k+\tilde{\gamma}_j)
(1+\tilde{\gamma}_j)}, \ \ k=1,\ldots, m,
\end{aligned}
\end{equation}
and $\tilde{v}_j=0$ for $3+n+m\leq j \leq 2m+2n+4$.

\section{Numerical examples}
In this section, some numerical examples are given to illustrate the results obtained in section 4. Following section 4, we consider a variable annuity with guaranteed maturity payment $\max\{F_0,F_T\}$ under the multi-layer expense strategy (1.2) with $B_1=F_0$ and $B_2 > F_0$, where $F_t$ is determined by (2.1), (2.6) and (2.7). For the sake of simplicity, we consider the case that $m=n=1$ in (2.2), which means that $X$ in (2.1) is a double exponential jump diffusion process. In the following, for this variable annuity, we will compute its fair fee rates $\alpha_1^{*}$ and $\alpha_2^{*}$ via (4.1), (4.9) and (4.12). Besides, with the obtained $\alpha_1^{*}$ and $\alpha_2^{*}$, numerical results on the total collected fees (see (4.3)) and the total time of deducting fees (see Remark 3.3) are also presented.

For numerical computing the above quantities through Laplace inversion, we choose the Euler inversion algorithm, which is first developed in Dubner and Abate (1968) and can be implemented easily. We remark that many papers use this algorithm and its extensions to do numerical Laplace inversion, see, e.g., Petrella (2004). For the convenience of the reader only, we give some important results on this algorithm. For a real function $f(\cdot)$ defined in $(0, \infty)$ and $T\neq 0$,
\begin{equation}
f(T)=\frac{e^{\tilde{A}/2}}{2T}Re(\hat{f}(\frac{\tilde{A}}{2T}))+\frac{e^{\tilde{A}/2}}{T}\sum_{k=1}^{\infty}(-1)^kRe(\hat{f}(
\frac{\tilde{A}+2k\pi i}{2T}))-e_d,
\end{equation}
where $\hat{f}(\cdot)$ is the Laplace transform of $f(\cdot)$, $\tilde{A}$ is a positive constant, $e_d$ is the
 discretization errors and $Re(x)$ means the real part of $x$. Moreover, if $|f(T)|\leq B$, then $|e_d|\leq B e^{-\tilde{A}}$; if $|f(T)|\leq BT$, then $|e_d|\leq 3 B Te^{-\tilde{A}}$ (see (5.29) in Abate and Whitt (1992)).

Note that, as functions of $T$, $\max\left\{\mathbb E\left[e^{-rT}F_T\right], \mathbb E\left[e^{-rT}(F_0-F_T)_+\right]\right\} \leq F_0$ and $\max\left\{\mathbb E\left[\int_{0}^{T}\textbf{1}_{\{U_t< b_1\}}dt\right], \mathbb E\left[\int_{0}^{T}\textbf{1}_{\{U_t \geq b_2\}}dt\right] \right\} \leq T$. Therefore, in the following numerical calculations, we set $\tilde{A}=20$, which is enough to control the discretization errors. Under the martingale measure, the values of the parameters unless stated otherwise are given in Table 1.
\begin{table}[!h]
    \centering
     \caption{Values of the parameters.}
    \begin{tabular}{cccccccccc}
        \hline
      parameter& $F_0$  & $\sigma$ & $\lambda$ & $p_1$ &$q_1$& $\eta_1$ & $\vartheta_1$ & $r$ \\
       value & 100& 0.2&1&0.5&0.5&15&15&0.05\\
      \hline
       \end{tabular}
  \end{table}
We remind the reader that the value of the parameter $\mu$ is computed such that $\psi(1)=r$ (the martingale condition), i.e., $\mu=r-\frac{\sigma^2}{2}-\lambda\left(\frac{p_1\eta_1}{\eta_1-1} + \frac{q_1\vartheta_1}{\vartheta_1 + 1}-1\right)$. In addition, all numerical calculations are implemented in MATLAB. For the sake of brevity, the two quantities: $\mathbb E\left[\int_{0}^{T}
\textbf{1}_{\{U_t < b_1\}}dt\right]$ and $\mathbb E\left[\int_{0}^{T}
\textbf{1}_{\{U_t \geq b_2\}}dt\right]$ with $b_1=\ln(\frac{B_1}{F_0})$ and $b_2=\ln(\frac{B_2}{F_0})$, are denoted respectively by $Ttime1$ and $Ttime2$ in  the following tables.

  \begin{table}[!h]
    \centering
     \caption{Fair fee rates $\alpha_1^{*}$ with respect to $B_2$, where $T=10$ and $\alpha_2^*=\alpha^*_1/2$.}
    \begin{tabular}{c|c|c|c|c|c|c|c}
        \hline
      $B_2$& $105$& $110$ & $120$ & $150$ & $200$ & $300$ & $1000$ \\
         \hline
       $\alpha_1^{*}$ & 0.016& 0.017&0.018& 0.022& 0.028& 0.038& 0.048\\
          \hline
       Total fees &9.34& 9.35&9.47& 9.83&10.42& 11.85&13.15\\
          \hline
       Ttime1&  4.44 &4.43&4.43&4.46&4.56&4.75& 4.94\\
          \hline
       Ttime2&5.03&4.58&3.83& 2.31&1.10&0.32&0.002\\
          \hline
       \end{tabular}
  \end{table}
First, we let $\alpha_2= \frac{1}{2}\alpha_1$ and study the relationship between the fair fee rate $\alpha_1^{*}$ and $B_2$. We summarize all the related results in Table 2. From the last column of Table 2, we see that the fair fee rate $\alpha^*_1$ for the case of $B_2=1000$ is $4.8\%$, which is too large to be used in practice. In addition, the total time of deducting fees (Ttime1+Ttime2) is about $4.94$. This means that the insurer will have no fees income during $5.06$ years (more than half of the variable annuity's maturity) in total, which would be not accepted by the insurer.  When $B_2$ equals $120$, the value of $\alpha_1^*$  becomes $1.8\%$ and the total time of deducting fees increases to $8.26$. Under this case (i.e., $B_2=120$), if the value of $F_t$ exceeds $120$, the fee rate imposed to the policyholder is only $0.009$ ($\alpha^*_1/2$). Therefore, the strategy that $B_2$ takes the value of $120$ is advisable. Of course, when $B_2$ decreases, the value of $\alpha^*_1$ also declines. However, from Table 2, we know that the value of $\alpha_1^*$ when $B_2= 105$ is almost the same as that when $B_2=120$. Besides, under these two cases (i.e., $B_2=105$ and $B_2=120$), the total fees charged by the insurer are nearly the same as well, i.e., the policyholder pays the same cost for the provided guarantee. And because of this, the policyholder  will prefer to the case that $B_2=120$. So, compared with $B_2=120$, the strategy of $B_2=105$ has less competitiveness. Therefore, in the following tables (except Table 4), we only consider that $B_2=120$.

We are interested in the connection of the fair fee rates $\alpha_2^*$, $\alpha_1^*$ and the maturity $T$. The corresponding results are given in Table 3, where once again we let $\alpha_2$ equal $\alpha_1/2$.
\begin{table}[!h]
    \centering
     \caption{Fair fee rates $\alpha_1^{*}$ with respect to $T$, where $B_2=120$ and $\alpha_2^*=\alpha_1^*/2$.}
    \begin{tabular}{c|c|c|c|c|c|c|c}
        \hline
      $T$& $1$& $3$ & $5$ & $7$ & $10$ & 12 & $15$\\
         \hline
       $\alpha_1^{*}$ &  0.366& 0.098& 0.051& 0.031&0.018& 0.013&0.009\\
          \hline
       Total fees &21.26&15.82 & 13.53&11.46&9.47& 8.20&7.08\\
          \hline
       Ttime1& 0.68 &1.66&2.54&3.32&4.43&  5.11& 6.11\\
          \hline
       Ttime2&0.07&0.60&1.37&2.29&3.83& 4.94&6.67\\
          \hline
       \end{tabular}
  \end{table}

Intuitively, when the maturity $T$ becomes large, the total time of deducting fees also grows, which yields a small fee rate. This intuition is confirmed by the second row of Table 3. We remark that the fact that the fair fee rate $\alpha_1^*$ decreases with $T$ encourages the policyholders to hold longer variable annuities in some sense. Besides, unlike the usual financial options (whose prices increase with their maturity), the cost of the guarantees embedded in VAs (i.e., the Total fees in Table 3) decreases with the maturity. One possible reason for this difference is that the cost for a financial option is charged at the beginning of the contract while that for the guaranteed benefit of a variable annuity is paid during the whole life of the policy. Of course, the insurer cannot deduct all costs at the inception of a variable annuity. Otherwise, the insured will pay too much money due to the  long maturity. In return, for a financial option (whose maturity is typically lower than one year), one cannot take the fee deducting method used in variable annuities, because the fair fee rate $\alpha_1^*$ amounts to $36.6$ percent when $T=1$ in Table 3. Moreover, numerical results in Table 4 also demonstrate that the fee rates are too large even when $B_2=100.1$ and $\alpha_2=\alpha_1$ (note that the case of $B_2=100.1$ and $\alpha_2=\alpha_1$ means that one charges fees by a fixed rate on matter what the value of $F_t$ is).

\begin{table}[!h]
    \centering
     \caption{Fair fee rates $\alpha_1^{*}$ with respect to $B_2$, where $T=1$.}
    \begin{tabular}{|cccc|cccc|}
    \hline
    \multicolumn{4}{|c|}{$\alpha^*_2=0.5 \alpha^*_1$} & \multicolumn{4}{|c|}{$\alpha^*_2=\alpha^*_1$}\\
        \hline
      $B_2$&100.1& 110&120&  $B_2$&100.1& 110&120\\
         \hline
       $\alpha_1^{*}$ & 0.206 & 0.282& 0.366 &  $\alpha_1^{*}$ &0.131& 0.197&  0.291 \\
          \hline
       Total fees & 15.09 &  17.49&  21.26 & Total fees & 12.23& 13.74&17.61 \\
          \hline
       Ttime1& 0.652& 0.646& 0.676  & Ttime1&0.626& 0.605&   0.641\\
          \hline
       Ttime2&  0.345& 0.159&0.070 & Ttime2& 0.370&0.165& 0.067 \\
          \hline
       \end{tabular}
  \end{table}

In short, short-term contracts and long-term contracts should be treated separately, and it is likely that a pricing or hedging approach, which is proper to short-term contracts, may be not suitable for long-term contracts.

\begin{table}[!h]
    \centering
     \caption{Fair fee rates $\alpha_1^{*}$ with respect to $\sigma$ and $r$, where $B_2=120$, $T=10$ and $\alpha^*_2=\alpha^*_1/2$.}
    \begin{tabular}{c|c|c|c|c|c}
        \hline
      $\sigma$& $0.1$& $0.15$ & $0.2$ & $0.25$ & $0.3$\\
         \hline
       $\alpha_1^{*}$ &0.005& 0.011&0.018&0.027&0.036\\
          \hline
       Total fees & 2.35&5.57&9.47&14.37&19.08\\
          \hline
      $r$& $0.4$& $0.45$ & $0.5$ & $0.55$ & $0.6$\\
         \hline
       $\alpha_1^{*}$ &0.026& 0.021&0.018&0.015&0.013\\
          \hline
       Total fees &13.84&11.13&9.47&7.85&6.75\\
          \hline
       \end{tabular}
  \end{table}

In Table 5, we consider how the fair fee rates depend on the volatility $\sigma$ and the risk free rate $r$, respectively. It is obvious that the higher (lower) the volatility $\sigma$ (the risk free rate $r$) is, the larger the fee rate $\alpha_1^*$ is. In addition, compared with the risk free rate $r$, the volatility $\sigma$ has a larger influence on the fee rate $\alpha_1^*$. It should be noted that the total fees decrease with the volatility $\sigma$. Especially, when $\sigma=0.1$, the total fees are too small (just $2.35$) and only account for $2.35\%$ of the initial premium. This result suggests that the more risk aversion the insured is, the less money he/she pays.

\begin{table}[!h]
    \centering
     \caption{Fair fee rates $\alpha_1^{*}$ with respect to $\eta_1$ and $\vartheta_1$, where $B_2=120$, $T=10$ and $\alpha^*_2=\alpha^*_1/2$.}
    \begin{tabular}{c|c|c|c|c|c|c}
        \hline
      $\eta_1$&6& 8&10 &15&20 &50\\
         \hline
       $\alpha_1^{*}$ &0.028&0.023&0.020&0.018&0.017& 0.016\\
          \hline
       Total fees & 15.13& 12.29&10.62&9.47&8.92& 8.36\\
          \hline
            $\vartheta_1$&6& 8&10 &15&20 &50\\
         \hline
       $\alpha_1^{*}$& 0.026&0.022& 0.020&0.018&0.017& 0.016\\
          \hline
       Total fees &13.58 &11.56& 10.52&9.47&8.95& 8.41\\
          \hline
       \end{tabular}
  \end{table}

Finally, in Table 6, we give the results about the sensitivity of the rate $\alpha_1$ with respect to changes in jump densities $\eta_1$ and $\vartheta_1$. From Table 6, one can see that the fair fee rate $\alpha_1^*$ is decrease with both $\vartheta_1$ and $\eta_1$. When the value of $\eta_1$ rises from $6$ to $50$, the total fees decline from $15.13$ to $8.36$. Thus the parameter $\eta_1$ has a significant effect on the total fees. Besides, one can draw a similar conclusion for the parameter $\vartheta_1$ from Table 6. In short, the jump risk has a large influence on the pricing of variable annuities and thus should be treated seriously.


\section{Conclusion}
In this paper, we have investigated  the problem of pricing variable annuities with a multi-layer expense strategy.
In theory, we have derived formulas for the Laplace transform of the distribution of a jump diffusion process with hyper-exponential jumps and three-valued drift. Applying these formulas, we compute the fair fee rate for a variable annuity with guaranteed minimum maturity benefit under the multi-layer expense strategy via Laplace inversion. Moreover, the total fees and the total time of charging fees are calculated as well. From the numerical results, we find that the fair fee rate is sensitive to the volatility and the risk-free rate. Therefore, a more interesting and challenging extension is considering the case that both the interest rate and the volatility are allowed to be stochastic. Such extension is left for future research.

\bigskip
\begin{appendix}
\medskip
\section{}
The matrix $Q_1$ in (3.6) is given by
\begin{equation}
Q_1=\left(\begin{array}{cc}
Q_{11}&Q_{12}
\end{array}
\right), \tag{A.1}
\end{equation}
where
\begin{equation}
Q_{11}=
\begin{pmatrix}
\begin{smallmatrix}
1&\tilde{\beta}_1&\frac{\vartheta_1}{\vartheta_1+\tilde{\beta}_1}&\cdots&
\frac{\vartheta_n}{\vartheta_n+\tilde{\beta}_1}&\frac{\eta_1}{\eta_1-\tilde{\beta}_1}&\cdots&
\frac{\eta_m}{\eta_m-\tilde{\beta}_1}\\
\vdots&\vdots&\vdots&\vdots&\vdots&\vdots&\vdots&\vdots&\\
1&\tilde{\beta}_{m+1}&\frac{\vartheta_1}{\vartheta_1+\tilde{\beta}_{m+1}}&\cdots&
\frac{\vartheta_n}{\vartheta_n+\tilde{\beta}_{m+1}}&\frac{\eta_1}{\eta_1-\tilde{\beta}_{m+1}}&\cdots&
\frac{\eta_m}{\eta_m-\tilde{\beta}_{m+1}}\\
-L^{\beta_1}&-\beta_1L^{\beta_1}&\frac{-\vartheta_1L^{\beta_1}}{\vartheta_1+\beta_{1}}&\cdots&
\frac{-\vartheta_nL^{\beta_1}}{\vartheta_n+\beta_{1}}&\frac{\eta_1L^{\beta_1}}{\beta_{1}-\eta_1}&\cdots&
\frac{\eta_mL^{\beta_1}}{\beta_{1}-\eta_m}\\
\vdots&\vdots&\vdots&\vdots&\vdots&\vdots&\vdots&\vdots&\\
-L^{\beta_{m+1}}&-\beta_{m+1}L^{\beta_{m+1}}&\frac{-\vartheta_1L^{\beta_{m+1}}}{\vartheta_1+\beta_{m+1}}&\cdots&
\frac{-\vartheta_nL^{\beta_{m+1}}}{\vartheta_n+\beta_{m+1}}&\frac{\eta_1L^{\beta_{m+1}}}{\beta_{m+1}-\eta_1}
&\cdots&
\frac{\eta_mL^{\beta_{m+1}}}{\beta_{m+1}-\eta_m}\\
-1&\gamma_1 &\frac{-\vartheta_1}{\vartheta_1-\gamma_1}&\cdots&
\frac{-\vartheta_n}{\vartheta_n-\gamma_{1}}&\frac{-\eta_1}{\eta_1+\gamma_{1}}
&\cdots&
\frac{-\eta_m}{\eta_m+\gamma_{1}}\\
\vdots&\vdots&\vdots&\vdots&\vdots&\vdots&\vdots&\vdots\\
-1&\gamma_{n+1}&\frac{-\vartheta_1}{\vartheta_1-\gamma_{n+1}}&\cdots&
\frac{-\vartheta_n}{\vartheta_n-\gamma_{n+1}}&\frac{-\eta_1}{\eta_1+\gamma_{n+1}}
&\cdots&
\frac{-\eta_m}{\eta_m+\gamma_{n+1}}\\
0&0&0&\cdots&0&0&\cdots&0\\
\vdots&\vdots&\vdots&\vdots&\vdots&\vdots&\vdots&\vdots\\
0&0&0&\cdots&0&0&\cdots&0\\
\end{smallmatrix}
\end{pmatrix},\tag{A.2}
\end{equation}
and

\begin{equation}
Q_{12}=
\begin{pmatrix}
\begin{smallmatrix}
0&0&0&\cdots&0&0&\cdots&0\\
\vdots&\vdots&\vdots&\vdots&\vdots&\vdots&\vdots&\vdots\\
0&0&0&\cdots&0&0&\cdots&0\\
1&\beta_1&\frac{\vartheta_1}{\vartheta_1+\beta_1}&\cdots&
\frac{\vartheta_n}{\vartheta_n+\beta_1}&\frac{\eta_1}{\eta_1-\beta_1}&\cdots&
\frac{\eta_m}{\eta_m-\beta_1}\\
\vdots&\vdots&\vdots&\vdots&\vdots&\vdots&\vdots&\vdots&\\
1&\beta_{m+1}&\frac{\vartheta_1}{\vartheta_1+\beta_{m+1}}&\cdots&
\frac{\vartheta_n}{\vartheta_n+\beta_{m+1}}&\frac{\eta_1}{\eta_1-\beta_{m+1}}&\cdots&
\frac{\eta_m}{\eta_m-\beta_{m+1}}\\
L^{\gamma_1}&-\gamma_1L^{\gamma_1}&\frac{\vartheta_1L^{\gamma_1}}{\vartheta_1-\gamma_{1}}&\cdots&
\frac{\vartheta_nL^{\gamma_1}}{\vartheta_n-\gamma_{1}}&\frac{\eta_1L^{\gamma_1}}{\eta_1+\gamma_{1}}&\cdots&
\frac{\eta_mL^{\gamma_1}}{\eta_m+\gamma_{1}}\\
\vdots&\vdots&\vdots&\vdots&\vdots&\vdots&\vdots&\vdots&\\
L^{\gamma_{n+1}}&-\gamma_{n+1}L^{\gamma_{n+1}}&\frac{\vartheta_1L^{\gamma_{n+1}}}{\vartheta_1-
\gamma_{n+1}}&\cdots&
\frac{\vartheta_nL^{\gamma_{n+1}}}{\vartheta_n-\gamma_{n+1}}&\frac{\eta_1L^{\gamma_{n+1}}}
{\eta_1+\gamma_{n+1}}&\cdots&
\frac{\eta_m L^{\gamma_{n+1}}}{\eta_m+\gamma_{n+1}}\\
-1&\hat{\gamma}_1 &\frac{-\vartheta_1}{\vartheta_1-\hat{\gamma}_1}&\cdots&
\frac{-\vartheta_n}{\vartheta_n-\hat{\gamma}_{1}}&\frac{-\eta_1}{\eta_1+\hat{\gamma}_{1}}
&\cdots&
\frac{-\eta_m}{\eta_m+\hat{\gamma}_{1}}\\
\vdots&\vdots&\vdots&\vdots&\vdots&\vdots&\vdots&\vdots&\\
-1&\hat{\gamma}_{n+1}&\frac{-\vartheta_1}{\vartheta_1-\hat{\gamma}_{n+1}}&\cdots&
\frac{-\vartheta_n}{\vartheta_n-\hat{\gamma}_{n+1}}&\frac{-\eta_1}{\eta_1+\hat{\gamma}_{n+1}}
&\cdots&
\frac{-\eta_m}{\eta_m+\hat{\gamma}_{n+1}}
\end{smallmatrix}
\end{pmatrix},\tag{A.3}
\end{equation}
with $L=e^{b_1-b_2}$.

\section{The proof of Theorem 3.1}
Before starting the derivation of Theorem 3.1, some notation is introduced first.
We set $\tilde{Y} = \{\tilde{Y}_t :=  X_t - \alpha_1 t; t\geq 0\}$ and $\hat{Y} = \{\hat{Y}_t :=  X_t - \alpha_2 t; t\geq 0\}$. The law of $\tilde{Y}$ ($\hat{Y}$) starting from $y$ and the corresponding expectation are denoted by $\tilde{\mathbb P}_y$ ($\hat{\mathbb P}_y$) and $\tilde{\mathbb E}_y$ ($\hat{\mathbb E}_y$), respectively. We write briefly $\tilde{\mathbb P}$ ($\hat{\mathbb P}$) and
$\tilde{\mathbb E}$ ($\hat{\mathbb E}$) when $y = 0$ . For $z \in (-\vartheta_1, \eta_1)$, the L\'evy exponents of $\tilde{Y}$ and $\hat{Y}$ are given respectively by $\tilde{\psi}(z)$ and $\hat{\psi}(z)$ in (3.3).
For any $c, C \in \mathbb R$, we define the following stopping times:
\begin{equation}
\begin{aligned}
&\tau_{c}^{-} := \inf\{t\geq 0: X_t \leq c\}, & \tau_{C}^{+} := \inf\{t\geq 0: X_t \geq C \},\\
&\tilde{\tau}_{c}^{-} := \inf\{t\geq 0: \tilde{Y}_t \leq c\}, & \tilde{\tau}_{C}^{+} := \inf\{t\geq 0: \tilde{Y}_t \geq C \},\\
&\hat{\tau}_{c}^{-} := \inf\{t\geq 0: \hat{Y}_t \leq c\}, & \hat{\tau}_{C}^{+} := \inf\{t\geq 0: \hat{Y}_t \geq C \},\\
&\kappa_c^-:=\inf\{t\geq 0: U_t \leq c \}, & \kappa_C^+ :=\inf\{t \geq 0: U_t \geq C\}.
\end{aligned}
\end{equation}

Besides, we recall the results on the solutions of one-sided and two-sided exit problems of $X$, $\tilde{Y}$ and $\hat{Y}$ in the following lemma. For their proofs, one can refer to, e.g.,  Yin et al. (2013) (see Lemmas 2.2, 2.4 and  2.5).
\renewcommand{\thelem}{B.\arabic{lem}}
\begin{lem}
(i) Consider any nonnegative measurable function $g$ such that $\int_{-\infty}^{0}g(c+y)e^{\vartheta_j y}dy < \infty$ for $j=1, 2, \ldots, n$. For $q>0$ and $x > c$, we have
\begin{equation}
\begin{split}
&\hat{\mathbb E}_x\left[e^{-q\hat{\tau}_{c}^{-}}g(\hat{Y}_{\hat{\tau}_{c}^{-}})\right] =\left(g(c), g_{\vartheta_1}(c), \ldots, g_{\vartheta_n}(c)\right)\hat{Q}^{-1}\left(\begin{array}{c}
e^{-\hat{\gamma}_{1}(x - c)}\\[0.3em]
\vdots\\
e^{-\hat{\gamma}_{n+1}(x - c)}
\end{array}\right),
\end{split}
\end{equation}
where
\begin{equation}
\hat{Q}=
\left(
\begin{array}{ccccc}
1 & \frac{\vartheta_1}{\vartheta_1-\hat{\gamma}_{1}} & \cdots&
\frac{\vartheta_n}{\vartheta_n-\hat{\gamma}_{1}}  \\
\vdots&\vdots&\vdots&\vdots\\
1 & \frac{\vartheta_1}{\vartheta_1-\hat{\gamma}_{n+1}} & \cdots&
\frac{\vartheta_n}{\vartheta_n-\hat{\gamma}_{n+1}}
\end{array}\right),
\end{equation}
and $g_{\vartheta_j}(c)=\int_{-\infty}^{0}g(c+y)\vartheta_je^{\vartheta_jy}dy$ for $j=1,2,\ldots,n$.

(ii) Consider any nonnegative measurable function $g$ such that $\int_{0}^{\infty}g(C+y)e^{-\eta_i y}dy < \infty$ for $i=1, 2, \ldots, m$. For $q > 0$ and $x < C$, we have
\begin{small}
\begin{equation}
\tilde{\mathbb E}_x\left[e^{-q\tilde{\tau}_{C}^{+}}g(\tilde{Y}_{\tilde{\tau}_{C}^{+}})\right] =\left(g(C),g_{\eta_1}(C),\ldots,g_{\eta_m}(C)\right)
\tilde{Q}^{-1}\left(\begin{array}{c}
e^{\tilde{\beta}_{1}(x - C)}\\
\vdots\\
e^{\tilde{\beta}_{m+1}(x-C)}\\
\end{array}\right),
\end{equation}
\end{small}
where
\begin{equation}
\tilde{Q}=
\left(
\begin{array}{ccccc}
1 & \frac{\eta_1}{\eta_1-\tilde{\beta}_{1}} &\cdots& \frac{\eta_m}{\eta_m-\tilde{\beta}_{1}}\\
\vdots&\vdots&\vdots&\vdots\\
1 & \frac{\eta_1}{\eta_1-\tilde{\beta}_{m+1}} &\cdots& \frac{\eta_m}{\eta_m-\tilde{\beta}_{m+1}}\\
\end{array}\right),
\end{equation}
and $g_{\eta_i}(C)=\int_{0}^{\infty}\eta_ie^{-\eta_i y}g(C+y)dy$ for $i=1,2,\ldots,m$.

(3) Consider  any nonnegative measurable function $g$ such that $\int_{0}^{\infty}g(C+y)e^{-\eta_i y}dy < \infty$ for $i=1, 2, \ldots, m$, and $\int_{-\infty}^{0}g(c+y)e^{\vartheta_j y}dy < \infty$ for $j=1, 2, \ldots, n$.  For $q > 0$ and $c< x < C$, we have
\begin{small}
\begin{equation}
\begin{split}
\mathbb E_x\left[e^{-q\tau}g(X_{\tau})\right]
=\left(g(C),g_{\eta_1}(C), \ldots,g_{\eta_m}(C),g(c),g_{\vartheta_1}(c),\ldots,g_{\vartheta_n}(c)\right)Q_{c,C}^{-1}R_{c,C}
\end{split}
\end{equation}
\end{small}
where $\tau=\min\{\tau_c^-,\tau_C^+\}$, the transpose of the vector $R_{c,C}$ is given by
\begin{equation}
R_{c,C}^{T}=\left(e^{\beta_{1}(x - C)},\ldots,e^{\beta_{m+1}(x-C)},e^{-\gamma_{1}(x - c)},\ldots,
e^{-\gamma_{n+1}(x - c)}
\right),
\end{equation}
and
\begin{small}
\begin{equation}
Q_{c,C}=
\left(
\begin{array}{cccccccc}
1 & \frac{\eta_1}{\eta_1-\beta_{1}} & \cdots& \frac{\eta_m}{\eta_m-\beta_{1}}&\bar{x}^{\beta_{1}}&\frac{\vartheta_1\bar{x}^{\beta_{1}}}
{\vartheta_1+\beta_{1}}&\cdots&
\frac{\vartheta_n\bar{x}^{\beta_{1}}}{\vartheta_n+\beta_{1}} \\
\vdots&\vdots&\vdots&\vdots&\vdots&\vdots&\vdots&\vdots\\
1 & \frac{\eta_1}{\eta_1-\beta_{m+1}}&\cdots& \frac{\eta_m}{\eta_m-\beta_{m+1}}&\bar{x}^{\beta_{m+1}}&\frac{\vartheta_1\bar{x}^{\beta_{m+1}}}{\vartheta_1+\beta_{m+1}}&\cdots&
\frac{\vartheta_n\bar{x}^{\beta_{m+1}}}{\vartheta_n+\beta_{m+1}}\\
\bar{x}^{\gamma_{1}}&\frac{\eta_1\bar{x}^{\gamma_{1}}}{\eta_1+\gamma_{1}}&\cdots&\frac{\eta_m
\bar{x}^{\gamma_{1}}}{\eta_m+\gamma_{1}}&1 & \frac{\vartheta_1}{\vartheta_1-\gamma_{1}}& \cdots&
\frac{\vartheta_n}{\vartheta_n-\gamma_{1}}\\
\vdots&\vdots&\vdots&\vdots&\vdots&\vdots&\vdots&\vdots\\
\bar{x}^{\gamma_{n+1}}&\frac{\eta_1\bar{x}^{\gamma_{n+1}}}{\eta_1+\gamma_{n+1}}&\cdots&\frac{\eta_m
\bar{x}^{\gamma_{n+1}}}{\eta_m+\gamma_{n+1}}&1 & \frac{\vartheta_1}{\vartheta_1-\gamma_{n+1}}& \cdots&
\frac{\vartheta_n}{\vartheta_n-\gamma_{n+1}}\\
\end{array}\right),
\end{equation}
\end{small}
with $\bar{x}=e^{c-C}$.
\end{lem}

Moreover, the expressions for the distributions of $\hat{S}_{e(q)}:=\sup_{0\leq t \leq e(q)}\hat{Y}_t$ and $\hat{I}_{e(q)}:=\inf_{0 \leq t \leq e(q)}\hat{Y}_t$ are also required and are given in the following Lemma B.2. For its derivation, we refer to Lemma 1 in Asmussen et al. (2004).
\begin{lem}
(1) For $s>0$,
\begin{equation}
\hat{\mathbb E}\left[e^{-s\hat{S}_{e(q)}}\right]=\prod_{i=1}^{m}\left(\frac{s+\eta_k}{\eta_k}\right)
\prod_{k=1}^{m+1}\left(\frac{\hat{\beta}_k}{s+\hat{\beta}_k}\right)=\sum_{k=1}^{m+1}\frac{\hat{C}_k}
{s+\hat{\beta}_k},
\end{equation}
and
\begin{equation}
\hat{\mathbb P}\left(\hat{S}_{e(q)}\in dy\right)=\sum_{i=1}^{m+1}\hat{C}_ie^{-\hat{\beta}_i y}dy, \ \ y\geq 0,
\end{equation}
where
\begin{equation}
\frac{\hat{C}_i}{\hat{\beta}_i}=\prod_{k=1}^{m}\left(\frac{\eta_k-\hat{\beta}_i}{\eta_k}\right)\prod_{k=1, k\neq i}^{m+1}\left(\frac{\hat{\beta}_k}{\hat{\beta}_k-\hat{\beta}_i}\right), \ \ i=1,\ldots, m+1.
\end{equation}
(2) For $s >0$ and $y\geq 0$,
\begin{equation}
\hat{\mathbb E}\left[e^{s\hat{I}_{e(q)}}\right]=\prod_{i=1}^{m}\left(\frac{s+\vartheta_k}{\vartheta_k}\right)
\prod_{k=1}^{m+1}\left(\frac{\hat{\gamma}_k}{s+\hat{\gamma}_k}\right)=\sum_{j=1}^{n+1}
\frac{\hat{D}_j}{s+\hat{\gamma}_j},
\end{equation}
and
\begin{equation}
\hat{\mathbb P}\left(-\hat{I}_{e(q)}\in dy\right)=\sum_{j=1}^{n+1}\hat{D}_je^{-\hat{\gamma}_iy}dy,
\end{equation}
where
\begin{equation}
\frac{\hat{D}_j}{\hat{\gamma}_j}=\prod_{k=1}^{n}\left(\frac{\vartheta_k-\hat{\gamma}_j}{\vartheta_k}\right)
\prod_{k=1, k\neq i}^{n+1}\left(\frac{\hat{\gamma}_k}{\hat{\gamma}_k-\hat{\gamma}_i}\right), \ \ j=1,\ldots, n+1.
\end{equation}

\end{lem}
\renewcommand{\therem}{B.\arabic{rem}}
\begin{rem}
For the constants $\hat{C}_i$ in (B.11) and $\hat{D}_j$ in (B.14), we have
\begin{equation}
\sum_{i=1}^{m+1}\frac{\hat{C}_i}{\hat{\beta}_i}=1 \ \ and \ \  \sum_{j=1}^{n+1}\frac{\hat{D}_j}{\hat{\gamma}_j}=1,
\end{equation}
which can be proved by set $s=0$ in (B.9) and (B.12). More importantly, as both sides of the second equality in (B.9) are rational function of $s$, we can extend this identity to the whole complex plane except at $-\hat{\beta}_1$, $\ldots$, $-\hat{\beta}_{m+1}$. Then, we obtain the following result from the extended identity by letting $s=-\eta_k$:
\begin{equation}
\sum_{i=1}^{m+1}\frac{\hat{C}_i}{\hat{\beta}_i-\eta_k}=0, \ \ for  \ \ k=1, 2, \ldots, m.
\end{equation}
Similarly, from (B.12), we have
\begin{equation}
\sum_{j=1}^{n+1}\frac{\hat{D}_j}{\hat{\gamma}_j-\vartheta_k}=0, \ \ for \ \ k=1, 2, \ldots, n.
\end{equation}
\end{rem}

\begin{proof}\{The proof of Theorem 3.1\}
This proof is dividend into two steps with the purpose of making it clearly. For the first step, we omit some details as similar arguments have been used in Zhou and Wu (2015).

(i) First, for given $y > b_2 > b_1$, we introduce a function of $x$ as
\begin{equation}
J(x) := \mathbb P_x(U_{e(q)} > y).
\end{equation}

For $x < b_1$, we have
\begin{equation}
\begin{split}
&J(x) = \mathbb E_x\left[\textbf{1}_{\{U_{e(q)} > y\}} \textbf{1}_{\{e(q) > \kappa^+_{b_1}\}}\right]=
\tilde{\mathbb E}_x\left[e^{-q \tilde{\tau}_{b_1}^{+}}J(\tilde{Y}_{\tilde{\tau}_{b_1}^{+}})\right].
\end{split}
\end{equation}
From (B.4), we obtain that
\begin{equation}
J(x) = \sum_{i=1}^{m+1}E_ie^{\tilde{\beta}_i(x-b_1)},\ \ for \ \  x < b_1,
\end{equation}
with
\begin{equation}
\left(E_1,\ldots, E_{m+1}\right)=\left(J(b_1),J_{\eta_1}(b_1),\ldots, J_{\eta_m}(b_1)\right)\tilde{Q}^{-1}.
\end{equation}

For $b_1<x<b_2$, we can derive that
\begin{equation}
\begin{split}
J(x)
&=\mathbb E_x\left[e^{-q\kappa_{b_2}^+}\textbf{1}_{\{
\kappa_{b_2}^+<\kappa_{b_1}^-\}}J(U_{\kappa_{b_2}^+})\right]+\mathbb E_x\left[e^{-q\kappa_{b_1}^-}\textbf{1}_{\{
\kappa_{b_1}^-<\kappa_{b_2}^+\}}J(U_{\kappa_{b_1}^-})\right]\\
&=\mathbb E_x\left[e^{-q\tau_{b_2}^+}\textbf{1}_{\{
\tau_{b_2}^+<\tau_{b_1}^-\}}J(X_{\tau_{b_2}^+})\right]+\mathbb E_x\left[e^{-q\tau_{b_1}^-}\textbf{1}_{\{
\tau_{b_1}^-<\tau_{b_2}^+\}}J(X_{\tau_{b_1}^-})\right].
\end{split}
\end{equation}
It follows from (B.6) that
\begin{equation}
J(x)=\sum_{i=1}^{m+1}F_ie^{\beta_i(x-b_2)}+\sum_{j=1}^{n+1}G_je^{\gamma_j(b_1-x)}, \ \ for \ \  b_1<x<b_2,
\end{equation}
with
\begin{equation}
\begin{split}
&\left(F_1, \ F_2, \ \ldots,\ F_{m+1},\ G_1,\ G_2,\ \ldots,\ G_{n+1}\right)=\\
&\left(J(b_2),J_{\eta_1}(b_2), \ldots, J_{\eta_m}(b_2), J(b_1), J_{\vartheta_1}(b_1), \ldots, J_{\vartheta_n}(b_1)\right)Q_{b_1,b_2}^{-1}.
\end{split}
\end{equation}

For $x > b_2$, we can deduce that
\begin{equation}
\begin{split}
J(x)
&= \mathbb E_x\left[\textbf{1}_{\{U_{e(q) >y}\}}\textbf{1}_{\{e(q)< \kappa_{b_2}^-\}}\right]+ \mathbb E_x\left[\textbf{1}_{\{U_{e(q) > y}\}}\textbf{1}_{\{e(q)> \kappa_{b_2}^-\}}\right]\\
&=\hat{\mathbb E}_x\left[\textbf{1}_{\{\hat{Y}_{e(q) >y}\}}\textbf{1}_{\{\hat{I}_{e(q)} > b_2\}}\right]+ \hat{\mathbb E}_x\left[e^{-q\hat{\tau}_{b_2}^-}J(\hat{Y}_{\hat{\tau}_{b_2}^-})\right]\\
&=\int_{b_2-x}^{0}\hat{\mathbb P}(\hat{S}_{e(q)} > y - x -z)\hat{\mathbb P}(\hat{I}_{e(q)}\in dz)+\hat{\mathbb E}_x\left[e^{-q\hat{\tau}_{b_2}^-}J(\hat{Y}_{\hat{\tau}_{b_2}^-})\right].
\end{split}
\end{equation}
Applying (B.2) and Lemma B.2 to (B.25) leads to that
\begin{equation}
J(x)=\sum_{i=1}^{m+1}H_ie^{\hat{\beta}_i(x-y)}+\sum_{j=1}^{n+1}M_je^{\hat{\gamma}_j(b_2-x)}, \ \ for \ \ b_2 < x \leq y,
\end{equation}
and that (note that $\hat{\mathbb P}\left(\hat{S}_{e(q)}> t\right)=1$ for $t \leq 0$)
\begin{equation}
J(x)=1+\sum_{j=1}^{n+1}N_je^{\hat{\gamma}_j(y-x)}+\sum_{j=1}^{n+1}M_je^{\hat{\gamma}_j(b_2-x)}, \ \ for \ \ x \geq y,
\end{equation}
where
\begin{equation}
\begin{split}
&H_i=\frac{\hat{C}_i}{\hat{\beta}_i}\sum_{j=1}^{n+1}\frac{\hat{D}_j}{\hat{\beta}_i+\hat{\gamma}_j}, \ \ i=1,2,\ldots, m+1,\\
&N_j=\hat{D}_j\sum_{i=1}^{m+1}\frac{\hat{C}_i}{\hat{\beta_i}(\hat{\beta}_i+\hat{\gamma}_j)}-\frac{\hat{D}_j}
{\hat{\gamma}_j}, \ \ j=1,2,\ldots, n+1,
\end{split}
\end{equation}
and
\begin{small}
\begin{equation}
\begin{split}
&\left(M_1, \ldots, M_{n+1}\right)
=-\left(\sum_{i=1}^{m+1}\frac{\hat{D}_1\hat{C}_i e^{\hat{\beta}_i(b_2-y)}}
{\hat{\beta}_i(\hat{\beta}_i+\hat{\gamma}_1)}, \ldots, \sum_{i=1}^{m+1}\frac{\hat{D}_{n+1}\hat{C}_i e^{\hat{\beta}_i(b_2-y)}}
{\hat{\beta}_i(\hat{\beta}_i+\hat{\gamma}_{n+1})}\right)\\
&\ \ \ \ \ \ \ \ \ \ \ \ \ \ \ \ \  \ \ \ \ \ \ \ \ +\big(J(b_2), J_{\vartheta_1}(b_2), \ldots, J_{\vartheta_n}(b_2)\big)\hat{Q}^{-1}.
\end{split}
\end{equation}
\end{small}

The expression for $H_i$ in (3.5) can be obtained from (B.11), (B.12) and (B.28). Moreover,
from (B.28), we have
\begin{equation}
\begin{split}
&N_j=\hat{D}_j\sum_{i=1}^{m+1}\frac{\hat{C}_i}{\hat{\beta_i}(\hat{\beta}_i+\hat{\gamma}_j)}-\frac{\hat{D}_j}
{\hat{\gamma}_j}
=\frac{\hat{D}_j}{\hat{\gamma}_j}\sum_{i=1}^{m+1}\hat{C}_i
\left(\frac{1}{\hat{\beta_i}}-\frac{1}{(\hat{\beta}_i+\hat{\gamma}_j)}\right)-\frac{\hat{D}_j}
{\hat{\gamma}_j},
\end{split}
\end{equation}
which combined with (B.9) and (B.15), yields the expression of $N_j$ in (3.5). Finally, formula (3.6) will be derived in the second step.

(2) On one hand, from (B.8) and (B.24), we have
\begin{equation}
J_{\vartheta_k}(b_1)=\sum_{i=1}^{m+1}\frac{F_i\vartheta_k}{\vartheta_k+\beta_i}e^{\beta_i(b_1-b_2)}+
\sum_{j=1}^{n+1}\frac{G_j\vartheta_k}
{\vartheta_k-\gamma_j},\ \ k=1, 2, \ldots, n,
\end{equation}
and
\begin{equation}
J_{\eta_k}(b_2)=\sum_{i=1}^{m+1}\frac{F_i\eta_k}{\eta_k-\beta_i}+\sum_{j=1}^{n+1}\frac{G_j\eta_k}
{\eta_k+\gamma_j}e^{\gamma_j(b_1-b_2)}, \ \ k=1, 2, \ldots, m.
\end{equation}
On the other hand, applying (B.20) yields
\begin{equation}
J_{\vartheta_k}(b_1)=\int_{-\infty}^{0}J(b_1+y)\vartheta_ke^{\vartheta_k y}dy
=\sum_{i=1}^{m+1}\frac{E_i\vartheta_k}{\vartheta_k+\tilde{\beta}_i}, \ \ k=1, 2, \ldots, n.
\end{equation}
For $1\leq k \leq m$, using (B.26) and (B.27), we obtain
\begin{equation}
\begin{split}
J_{\eta_k}(b_2)
&=\int_{0}^{\infty}J(b_2+z)\eta_ke^{-\eta_k z}dz
=e^{\eta_k(b_2-y)} +\sum_{j=1}^{n+1}\frac{N_j\eta_k}{\eta_k+\hat{\gamma}_j}e^{\eta_k(b_2-y)}\\
&\ \ +\sum_{i=1}^{m+1}\frac{H_i\eta_k}{\hat{\beta}_i-\eta_k}\left(e^{\eta_k(b_2-y)}-
e^{\hat{\beta}_i(b_2-y)}\right)+\sum_{j=1}^{n+1}\frac{M_j\eta_k}{\eta_k+\hat{\gamma}_j}.
\end{split}
\end{equation}

Therefore, from (B.31) $\sim$ (B.34), we immediately obtain that
\begin{equation}
\sum_{i=1}^{m+1}\frac{E_i\vartheta_k}{\vartheta_k+\tilde{\beta}_i}=\sum_{i=1}^{m+1}\frac{F_i\vartheta_k}
{\vartheta_k+\beta_i}e^{\beta_i(b_1-b_2)}+\sum_{j=1}^{n+1}\frac{G_j\vartheta_k}{\vartheta_k-\gamma_j}, \ \ k=1,2, \ldots, n,
\end{equation}
and that, for $1\leq k \leq m$,
\begin{small}
\begin{equation}
\sum_{i=1}^{m+1}\frac{F_i\eta_k}{\eta_k-\beta_i}+\sum_{j=1}^{n+1}\frac{G_j\eta_k}{\eta_k+\gamma_j}
e^{\gamma_j(b_1-b_2)}=\sum_{j=1}^{n+1}\frac{M_j\eta_k}{\eta_k+\hat{\gamma}_j}-\sum_{i=1}^{m+1}\frac{
H_i\eta_k}{\hat{\beta}_i-\eta_k}e^{\hat{\beta}_i(b_2-y)}.
\end{equation}
\end{small}
To derive (B.36), we have used the following identity:
\begin{equation}
\sum_{i=1}^{m+1}\frac{H_i\eta_k}{\hat{\beta}_i-\eta_k}
+1
+\sum_{j=1}^{n+1}\frac{N_j\eta_k}{\eta_k+\hat{\gamma}_j}=0,
\end{equation}
which can be proved as following:
\begin{equation}
\begin{split}
&\sum_{i=1}^{m+1}\frac{H_i\eta_k}{\hat{\beta}_i-\eta_k}
+\sum_{j=1}^{n+1}\frac{N_j\eta_k}{\eta_k+\hat{\gamma}_j}\\
&=\sum_{i=1}^{m+1}\frac{\hat{C}_i}{\beta_i}
\sum_{j=1}^{n+1}\frac{\hat{D}_j\eta_k}{\eta_k+\hat{\gamma}_j}\left(\frac{1}{\hat{\beta}_i-\eta_k}-
\frac{1}{\hat{\beta}_i+\hat{\gamma}_j}\right)+\sum_{j=1}^{n+1}\frac{N_j\eta_k}{\eta_k+\hat{\gamma}_j}\\
&=\left(\sum_{i=1}^{m+1}\frac{\hat{C}_i}{\hat{\beta}_i-\eta_k}-\sum_{i=1}^{m+1}\frac{\hat{C}_i}{\hat{\beta}_i}
\right)\sum_{j=1}^{n+1}\frac{\hat{D}_j\eta_k}{\eta_k+\hat{\gamma}_j}-\sum_{j=1}^{n+1}\frac{\hat{D}_j\eta_k}
{\eta_k+\hat{\gamma}_j}=-1,
\end{split}
\end{equation}
where the second equality follows from (B.28) and the third one is due to (B.15) and (B.16).

Next, it follows from (B.3), (B.5), (B.21) and (B.29) that
\begin{equation}
\begin{split}
&J_{\eta_k}(b_1)=\sum_{i=1}^{m+1}\frac{E_i\eta_k}{\eta_k-\tilde{\beta}_i}, \ \ k=1, 2, \ldots, m,
\end{split}
\end{equation}
and that
\begin{equation}
\begin{split}
&J_{\vartheta_k}(b_2)
=\sum_{j=1}^{n+1}\left(M_j+\sum_{i=1}^{m+1}\frac{\hat{C}_i\hat{D}_j}{
\hat{\beta}_i(\hat{\beta}_i+\hat{\gamma}_j)}e^{\hat{\beta}_i(b_2-y)}\right)\frac{\vartheta_k}
{\vartheta_k-\hat{\gamma}_j}\\
&=\sum_{j=1}^{n+1}\left(M_j+\sum_{i=1}^{m+1}\frac{\hat{C}_i\hat{D}_j\vartheta_k}{\hat{\beta}_i(\vartheta_k+
\hat{\beta}_i)}\left(\frac{1}{\hat{\beta}_i+\hat{\gamma}_j}+\frac{1}{\vartheta_k-\hat{\gamma}_j}\right)
e^{\hat{\beta}_i(b_2-y)}\right)\\
&=\sum_{j=1}^{n+1}\frac{M_j \vartheta_k}
{\vartheta_k-\hat{\gamma}_j}+\sum_{i=1}^{m+1}\frac{H_i\vartheta_k}{\vartheta_k+\hat{\beta}_i}
e^{\hat{\beta}_i(b_2-y)}, \ \ k=1, 2, \ldots, n.
\end{split}
\end{equation}
where the second equality can be verified by using (B.17) and (B.28).

Using (B.20), (B.23), (B.26) and (B.27), one can derive that
\begin{small}
\begin{equation}
\begin{split}
&J_{\vartheta_k}(b_2)
=\int_{-\infty}^{0}J(b_2+y)\vartheta_ke^{\vartheta_k y}dy=\sum_{j=1}^{n+1}\frac{G_j\vartheta_k}{
\vartheta_k-\gamma_j}\left(e^{\gamma_j(b_1-b_2)}-e^{\vartheta_k(b_1-b_2)}\right)\\
&+\sum_{i=1}^{m+1}\frac{E_i\vartheta_k}{\vartheta_k+\tilde{\beta}_i}
e^{\vartheta_k(b_1-b_2)}+\sum_{i=1}^{m+1}\frac{F_i\vartheta_k}{\vartheta_k+\beta_i}\left(1-
e^{(\vartheta_k+\beta_i)(b_1-b_2)}\right), \ 1 \leq k \leq n,
\end{split}
\end{equation}
\end{small}
and that, for $k=1, \ldots, m$,
\begin{small}
\begin{equation}
\begin{split}
&J_{\eta_k}(b_1)=\int_{0}^{\infty}J(b_1+z)\eta_ke^{-\eta_k z}dz=\sum_{j=1}^{n+1}\frac{M_j\eta_k}{\eta_k+\hat{\gamma}_j}e^{\eta_k(b_1-b_2)}
+e^{\eta_k(b_1-y)}\\
&+\sum_{i=1}^{m+1}\frac{F_i\eta_k}{\beta_i-\eta_k}\left(e^{\eta_k(b_1-b_2)}-e^{\beta_i(b_1-b_2)}
\right)+\sum_{j=1}^{n+1}\frac{G_j\eta_k}{\eta_k+\gamma_j}\left(1-e^{(\eta_k+\gamma_j)(b_1-b_2)}\right)\\
&+
\sum_{i=1}^{m+1}\frac{H_i\eta_k}{\hat{\beta}_i-\eta_k}\left(e^{\eta_k(b_1-y)}-e^{\eta_k(b_1-b_2)}
e^{\hat{\beta}_i(b_2-y)}\right)+\sum_{j=1}^{n+1}\frac{N_j\eta_k}{\eta_k+\hat{\gamma}_j}e^{\eta_k(b_1-y)}.
\end{split}
\end{equation}
\end{small}

Thus, for $1 \leq k \leq n$, applying formulas (B.35), (B.40) and (B.41) leads to
\begin{small}
\begin{equation}
\begin{split}
&\sum_{j=1}^{n+1}\frac{M_j\vartheta_k}{\vartheta_k-\hat{\gamma}_j}+\sum_{i=1}^{m+1}
\frac{H_i\vartheta_k}{\vartheta_k+\hat{\beta}_i}e^{\hat{\beta}_i(b_2-y)}
=\sum_{i=1}^
{m+1}\frac{F_i\vartheta_k}{\vartheta_k+\beta_i}+\sum_{j=1}^{n+1}\frac{G_j\vartheta_k}{\vartheta_k-\gamma_j}
e^{\gamma_j(b_1-b_2)},
\end{split}
\end{equation}
\end{small}
and for $k=1,2, \ldots, m$, applying (B.36), (B.37), (B.39) and (B.42) leads to
\begin{equation}
\sum_{i=1}^{m+1}\frac{E_i\eta_k}{\eta_k-\tilde{\beta}_i}=\sum_{i=1}^{m+1}\frac{F_i\eta_k}{\eta_k-\beta_i}
e^{\beta_i(b_1-b_2)}+\sum_{j=1}^{n+1}\frac{G_j\eta_k}{\eta_k+\gamma_j}.
\end{equation}

Finally, from part (1) of Lemma 2.1, we know that $J(x)$ is continuously differentiable on $\mathbb R$. This means that $J(b_1-)=J(b_1+)$, $J(b_2-)=J(b_2+)$, $J^{\prime}(b_1-)=J^{\prime}(b_1+)$ and $J^{\prime}(b_2-)=J^{\prime}(b_2+)$, which combined with (B.20), (B.23) and (B.26), yields
\begin{small}
\begin{equation}
\begin{split}
&\sum_{i=1}^{m+1}E_i=\sum_{i=1}^{m+1}F_ie^{\beta_i(b_1-b_2)}+\sum_{j=1}^{n+1}G_j,\\
&\sum_{i=1}^{m+1}E_i\tilde{\beta}_i=\sum_{i=1}^{m+1}F_i\beta_i e^{\beta_i(b_1-b_2)}-
\sum_{j=1}^{n+1}G_j\gamma_j,\\
&\sum_{i=1}^{m+1}F_i+\sum_{j=1}^{n+1}G_je^{\gamma_j(b_1-b_2)}=\sum_{i=1}^{m+1}H_ie^{\hat{\beta}_i(b_2-y)}
+\sum_{j=1}^{n+1}M_j,\\
&\sum_{i=1}^{m+1}F_i\beta_i-\sum_{j=1}^{n+1}G_j\gamma_je^{\gamma_j(b_1-b_2)}
=\sum_{i=1}^{m+1}H_i\hat{\beta}_ie^{\hat{\beta}_i(b_2-y)}
-\sum_{j=1}^{n+1}M_j\hat{\gamma}_j.
\end{split}
\end{equation}
\end{small}
Therefore, from (B.35), (B.36), (B.43), (B.44) and (B.45), we deduce (3.6). This completes the proof.
\end{proof}


\section{}
\begin{proof}\{The proof of Corollary 3.1\}
It follows from (3.4) and (3.14) that
\begin{small}
\begin{equation}
\begin{split}
&\mathbb P_x \left(U_{e(q)} > b_2\right)-\mathbb P_x\left(U_{e(q)} \geq b_2\right)=\lim_{y \downarrow b_2}\mathbb P_x\left(U_{e(q)} > y \right)-\lim_{y \uparrow b_2}\mathbb P_x\left(U_{e(q)} > y \right)\\
=&
\left\{
\begin{aligned}
& \sum_{i=1}^{m+1}\left(E^0_i-\hat{E}^0_i\right)e^{\tilde{\beta}_i(x-b_1)}, & x \leq b_1,\\
& \sum_{i=1}^{m+1}\left(F^0_i-\hat{H}^0_i-\hat{U}_i\right)e^{\beta_i(x-b_2)}
+\sum_{j=1}^{n+1}\left(G^0_j-\hat{G}^0_j\right)e^{\gamma_j(b_1-x)}, & b_1 \leq  x \leq b_2, \\
&\sum_{j=1}^{n+1}\left(N_j + M^0_j-\hat{N}^0_j\right)e^{\hat{\gamma}_j(b_2-x)}, & x \geq b_2,
\end{aligned}
\right.
\end{split}
\end{equation}
\end{small}
with
\begin{equation}
\begin{split}
(\hat{E}^0_1,\ldots,\hat{E}^0_{m+1},\hat{H}^0_1,\ldots,\hat{H}^0_{m+1},\hat{G}^0_1,\ldots,\hat{G}^0_{n+1}
,\hat{N}^0_1, \ldots, \hat{N}^0_{n+1})Q_1=\hat{h}^0,\\
\left(E^0_1,\ldots,E^0_{m+1},F^0_1,\ldots,F^0_{m+1},G^0_1,\ldots,G^0_{n+1}
,M^0_1, \ldots, M^0_{n+1}\right)Q_1=h^0,
\end{split}
\end{equation}
where $h^0$ and $\hat{h}^0$ are given by $h$ in (3.7) and $\hat{h}$ in (3.17) with $y=b_2$, respectively.

Similar to derive (B.37), we can obtain the following equalities by using (B.15), (B.17) and (B.28):
\begin{equation}
\begin{split}
&\sum_{i=1}^{m+1}H_i=1+\sum_{j=1}^{n+1}N_j,\\ &\sum_{i=1}^{m+1}H_i\hat{\beta}_i+\sum_{j=1}^{n+1}N_j\hat{\gamma}_j=0,\\
&\sum_{i=1}^{m+1}\frac{H_i\vartheta_k}{\vartheta_k+\hat{\beta}_i}-\sum_{j=1}^{n+1}\frac{N_j\vartheta_k}
{\vartheta_k-\hat{\gamma}_j}-1=0,  \ \ for \ \  k=1,2,\ldots, n.
\end{split}
\end{equation}

From (3.20), (B.37), (C.2), (C.3) and (A.1), we can verify easily that
\begin{equation}
E^0_i=\hat{E}^0_i, \ \ F^0_i-\hat{H}^0_i-\hat{U}_i=0,\ \ G^0_j-\hat{G}^0_j=0, \ \ N_j + M^0_j-\hat{N}^0_j=0,
\end{equation}
i.e.,
\begin{equation}
\begin{split}
(\underbrace{0,\ldots,0}_{m+1},\hat{U}_1,\ldots,\hat{U}_{m+1},\underbrace{0,\ldots,0}_{n+1}
,-N_1, \ldots, -N_{n+1})Q_1=h^0-\hat{h}^0.
\end{split}
\end{equation}

Combining (C.1) with (C.4) leads to
\begin{equation}
\mathbb P_x\left(U_{e(q)}=b_2\right)=0.
\end{equation}
The proof of $\mathbb P_x\left(U_{e(q)}=b_1\right)=0$ is similar, thus we omit the details.
\end{proof}

\begin{proof}\{The proof of Corollary 3.2\}
For fixed $b_1$, letting $b_2 \downarrow b_1$ in (3.4), we immediately deduce (3.23) with
$E^1_i:=\lim_{b_2 \downarrow b_1} E_i$ and $M_j^1:=\lim_{b_2 \downarrow b_1} M_j$.

From (B.35) and (B.43) with $b_2 \downarrow b_1$, we can obtain
\begin{equation}
\sum_{j=1}^{n+1}\frac{M^1_j}{\vartheta_k-\hat{\gamma}_j}+\sum_{i=1}^{m+1}\frac{H_i}{\vartheta_k+\hat{\beta}_i}e^{
\hat{\beta}_i(b_1-y)}=\sum_{i=1}^{m+1}\frac{E^1_i}
{\vartheta_k+\tilde{\beta}_i}, \ \ 1 \leq k \leq n.
\end{equation}
Besides, it following from (B.36) and (B.44) with $b_2=b_1$ that
\begin{equation}
\sum_{i=1}^{m+1}\frac{E^1_i}{\eta_k-\tilde{\beta}_i}=\sum_{j=1}^{n+1}\frac{M^1_j}{\eta_k+
\hat{\gamma}_j}-\sum_{i=1}^{m+1}\frac{H_i}{\hat{\beta}_i-\eta_k}e^{\hat{\beta}_i(b_1-y)}, \ \
1\leq k \leq m.
\end{equation}
Applying (B.45) with $b_2=b_1$ produces
\begin{equation}
\begin{split}
&\sum_{i=1}^{m+1}E^1_i=\sum_{i=1}^{m+1}H_ie^{\hat{\beta}_i(b_1-y)}+\sum_{j=1}^{n+1}M_j^1,\\
&\sum_{i=1}^{m+1}E^1_i\tilde{\beta}_i=\sum_{i=1}^{m+1}H_i\hat{\beta}_ie^{\hat{\beta}_i(b_1-y)}
-\sum_{j=1}^{n+1}M_j^1\hat{\gamma}_j.
\end{split}
\end{equation}

Next, we want to solve (C.7), (C.8) and (C.9). First,
define a function of $x$ as
\begin{equation}
f(x)=\sum_{i=1}^{m+1}\frac{E^1_i}{x-\tilde{\beta}_i}-\sum_{j=1}^{n+1}\frac{M_j^1}{x+\hat{\gamma}_j}
-\sum_{i=1}^{m+1}\frac{H_i}{x-\hat{\beta}_i}e^{\hat{\beta}_i(b_1-y)}.
\end{equation}

From (C.7) and (C.8), we will obtain that
$ f(-\vartheta_k)=0$ for $1\leq k \leq n$ and $f(\eta_k)=0$ for $1\leq k \leq m$.
Combining these results with (C.9), we obtain that
\begin{equation}
f(x)=\frac{\prod_{i=1}^{m}(x-\eta_i)\prod_{j=1}^{n}(x+\vartheta_j)}{\prod_{i=1}^{m+1}(x-\tilde{\beta}_i)
\prod_{j=1}^{n+1}(x+\hat{\gamma}_j)}\frac{l_m x^m+ l_{m-1} x^{m-1}+\cdots + l_0}{
\prod_{i=1}^{m+1}(x-\hat{\beta}_i)},
\end{equation}
for some proper constants $l_m$, $l_{m-1}$, $\ldots$, $l_0$. Furthermore, for $1\leq i \leq m+1$, it follows from the definition (C.10) that
\begin{equation}
\lim_{x\rightarrow \hat{\beta}_i}f(x)(x-\hat{\beta}_i)=-H_ie^{\hat{\beta}_i(b_1-y)}.
\end{equation}
From (C.11) and (C.12), we conclude that $f(x)$ has another form as following:
\begin{equation}
f(x)=\frac{\prod_{i=1}^{m}(x-\eta_i)\prod_{j=1}^{n}(x+\vartheta_j)}{\prod_{i=1}^{m+1}(x-\tilde{\beta}_i)
\prod_{j=1}^{n+1}(x+\hat{\gamma}_j)}\sum_{i=1}^{m+1}\frac{\prod_{k=1}^{m+1}(\hat{\beta}_i-\tilde{\beta}_k)
\prod_{k=1}^{n+1}(\hat{\beta}_i+\hat{\gamma}_k)}{\prod_{k=1}^{m}(\hat{\beta}_i-\eta_k)
\prod_{k=1}^{n}(\hat{\beta}_i+\vartheta_k)}\frac{-H_i}{x-\hat{\beta}_i}e^{\hat{\beta}_i(b_1-y)}
.
\end{equation}

Therefore, from (C.10), (C.13) and the expression of $H_i$ in (3.5), we can derive (3.24) and (3.25). The derivation of (3.26) from Theorem 3.2 is very similar, thus we omit the details. For (3.29), one can obtain  it by using a similar idea to that in the proof of Corollary 3.1. However, as this process involves some computations, we give the details for the convenience of the reader.

From (3.23) and (3.26), we can show that
\begin{equation}
\begin{split}
&\mathbb P_x\left(U_{e(q)} > b_1\right)+\mathbb P_x\left(U_{e(q)} < b_1\right)\\
=&\lim_{y \downarrow b_1}\mathbb P_x\left(U_{e(q)} > y\right)+\lim_{y \uparrow b_1}\mathbb P_x\left(U_{e(q)} < y\right)\\
=&\begin{small}
\left\{
\begin{aligned}
&1 + \sum_{i=1}^{m+1}\left(\tilde{E}_i + \tilde{F}^{1,0}_i + E_i^{1,0} \right)e^{\tilde{\beta}_i(x- b_1)}, & x\leq b_1,\\
&1+\sum_{j=1}^{n+1}\left(\tilde{N}^{1,0}_j+N_j+M_j^{1,0}\right)e^{\hat{\gamma}_j(b_1-x)}, & x \geq b_1,
\end{aligned}
\right.
\end{small}
\end{split}
\end{equation}
where $M_j^{1,0}$, $E_i^{1,0}$, $\tilde{N}_j^{1,0}$ and $\tilde{F}_i^{1,0}$ are given respectively by $M^1_j$, $E_i^1$, $\tilde{N}_j^1$ and $\tilde{F}_i^1$ with $y=b_1$. From (3.5), (3.10), (3.24), (3.25), (3.27) and (3.28), we can obtain the following result by using Lemma C.1:
\begin{equation}
\tilde{E}_i + \tilde{F}^{1,0}_i + E_i^{1,0}=0, \ \ \tilde{N}^{1,0}_j+N_j+M_j^{1,0}=0.
\end{equation}
Therefore, formulas (C.14) and (C.15) lead to (3.29).
\end{proof}
\renewcommand{\thelem}{C.1}
\begin{lem}
For distinct constants $\tilde{l}_1, \ldots, \tilde{l}_{n_1}$ and arbitrary constants $\hat{l}_1, \ldots, \hat{l}_{m_1}$ with $m_1< n_1-1$, we have
\begin{equation}
\sum_{i=1}^{n_1}\frac{\prod_{k=1}^{m_1}(\tilde{l}_i-\hat{l}_k)}{\prod_{k=1,k\neq i}^{n_1}(\tilde{l}_i-\tilde{l}_k)}=0.
\end{equation}
\end{lem}
\begin{proof}
The proof is easy by noting that the left-hand side of (C.16) is the coefficient of $x^{n_1-1}$ in the numerator of the following rational function:
\begin{equation}
\frac{\prod_{i=1}^{m_1}(x-\hat{l}_i)}{\prod_{i=1}^{n_1}(x-\tilde{l}_i)}=
\sum_{i=1}^{n_1}\frac{\prod_{k=1}^{m_1}(\tilde{l}_i-\hat{l}_k)}{\prod_{k=1,k\neq i}^{n_1}(\tilde{l}_i-\tilde{l}_k)}\frac{1}{x-\tilde{l}_i}.
\end{equation}
\end{proof}

\end{appendix}

\end{document}